\documentclass{article}
\usepackage[comma,square,sort&compress, numbers]{natbib}
\usepackage[greek, english]{babel}
\usepackage{amsbsy,amssymb,amscd,amsmath,amsthm,amsfonts}
\usepackage{eurosym}
\usepackage{epsf}
\usepackage{epsfig}
\usepackage{color}
\usepackage{bbold}
\usepackage[T1]{fontenc}
\usepackage[utf8]{inputenc}
\usepackage{authblk}
\usepackage{enumitem}

\newfont{\gothic}{eufm10}

\def\Z{{\mathbb{Z}}}                   \def\R{{\RR}}
\def\RR{{\mathbb{R}}}        \def\N{{\mathbb{N}}}

        \newtheorem{theorem}{Theorem}[section]
\newtheorem{lemma}[theorem]{Lemma}
\newtheorem{proposition}[theorem]{Proposition}
\newtheorem*{propositiona}{Proposition 2.1}
\newtheorem{corollary}[theorem]{Corollary}
\newtheorem{definition1}[theorem]{Definition}
\newtheorem*{theorema}{Theorem (Existence of asymptotically defined minimisers)}

\newenvironment{definition}{\begin{definition1}\rm}{\end{definition1}}

\newtheorem{remark1}[theorem]{Remark}
\newenvironment{remark}{\begin{remark1}\rm}{\end{remark1}}

\newtheorem{example1}[theorem]{Example}

\def\barray{\begin{eqnarray*}}             \def\earray{\end{eqnarray*}}
\def\beq{\begin{equation}} \def\eeq{\end{equation}}

\DeclareMathOperator{\out}{out}
\DeclareMathOperator{\inn}{in}
\DeclareMathOperator{\f}{f}

\makeatletter 
\title{Minimisers of the Allen-Cahn equation on hyperbolic graphs}  
\author{Bla\v z Mramor } 
\affil{\small University of Freiburg, Germany. }
\date{}

\begin{document}  
\hyphenation{asympto-ti-cally}
\hyphenation{ma-ni-folds}

\newcommand{\p}{\partial}
\maketitle

\noindent

\abstract{\noindent 
We investigate minimal solutions of the Allen-Cahn equation on a Gromov-hyperbolic graph. Under some natural conditions on the graph, we show the existence of non-constant uniformly-bounded minimal solutions with prescribed asymptotic behaviours. For a phase field model on a hyperbolic graph, such solutions describe energy-minimising steady-state phase transitions that converge towards prescribed phases given by the asymptotic directions on the graph.}

\section{Introduction}

In the first subsection of the introduction we briefly state our results and in the second we give a motivation for them, together with a discussion of related literature.

\subsection{Setting and results}
We are interested in the Allen-Cahn equation on an infinite connected graph $\Gamma=(G,E)$, where $G$ is the set of vertices and $E$ the set of edges. We assume that $\Gamma$ is locally uniformly bounded and that it furthermore satisfies the following conditions, the precise statements of which can be found in section \ref{geometric_graphs}. By setting the length of edges to $1$ and denoting by $d$ the geodesic distance, $(\Gamma,d)$ is a geodesically complete unbounded metric space and we assume that it is $\delta$-hyperbolic in the sense of Gromov. The space $(\Gamma,d)$ can be then compactified by adding the {\it boundary at infinity $\p \Gamma$}, where the points in $\p \Gamma$ are equivalence classes of geodesic rays. We assume that there exists a vertex $v\in G$, such that every point is uniformly close to a geodesic ray from $v$. Finally, we assume that $\Gamma$ satisfies a particular isoperimetric condition of exponent smaller than one (given by (\ref{IP}) in section \ref{isoperimetric}), which controls the local growth of the space. 

%We study the Allen-Cahn equation on $G\subset \Gamma$. 
To describe the Allen-Cahn equation, let $V:\R\to \R$ be a $C^2$ double-well potential, i.e. assume that $V$ has at least two consecutive non-degenerate absolute minima $c_0,c_1\in \R$. A typical example is $\displaystyle V(s):=(1-s^2)^2$, or a generic (in the sense of Morse theory) periodic function. Denote for every function $x:G\to\R$ the discrete Laplace operator by $$\Delta_g(x):=\sum_{d(\tilde g,g)=1}(x_{\tilde g}-x_g) \ .$$ The Allen-Cahn equation is then given by \begin{equation}\label{rr} \Delta_g(x)-V'(x_g)=0 , \text{ for all } \ g\in G .\tag{AC}\end{equation} The equation (\ref{rr}) comes with a variational structure given by the formal action functional \begin{equation}\label{fe}W(x):=\sum_{g\in G}\left(\frac{1}{4}\sum_{d(\tilde g,g)=1}(x_{\tilde g}-x_g)^2+V(x_g)\right)=\sum_{g\in G}\left(\frac{1}{4}|\nabla_g(x)|^2+V(x_g)\right),\tag{W}\end{equation} which has a well defined gradient $\nabla_g(W(x))=\Delta_g(x)-V'(x_g)$ (see also section \ref{var_prob}). As is usual in the calculus of variation, we call an entire solution $x:G\to \R$ of (\ref{rr}) a global minimiser, if it minimises the action (\ref{fe}) globally, i.e., if for every finitely supported variation $\tilde x:G\to \R$ of $x$ $$W(x+\tilde x) \geq W(x).$$ 

The goal of this paper is to construct a large class of uniformly bounded global minimisers of (\ref{rr}), which have prescribed asymptotic behaviour, tending to either $c_0$ or $c_1$ along infinite geodesic rays in $G$. 
To state our result, we note that with respect to a visual metric (explained in section \ref{compactifying}) $\p\Gamma$ and $\Gamma\cup \p\Gamma$ are compact metric spaces. We prove the so-called {\it minimal Dirichlet problem at infinity} (c.f. definition \ref{dp_def}):

\begin{theorema}
Let $D_0, D_1\subset \p \Gamma$ be such that $\overline{D_0 \cup D_1}=\p \Gamma$, $\mathring{D}_0 \cap \mathring{D}_1=\varnothing$ and $\overline{\mathring{D}}_j=D_j$ for $j\in \{0,1\}$. Then there exists a global minimiser $x$, such that for every $\varepsilon>0$ and $j\in \{0,1\}$ there exists an open set $\mathcal{O}_{j}\subset \Gamma\cup \p\Gamma$ (with respect to the topology induced by the visual metric), such that $\mathring{D}_j \subset \mathcal{O}_{j}$ and such that for every $g\in G\cap \mathcal{O}_j$, $$|x_g-c_j|\leq \varepsilon.$$
\end{theorema}

\subsection{Motivation}

A time-dependent version of the Allen-Cahn equation first appeared in the study of phase field models on $\R^2$ describing a mathematical model of phase transitions \cite{AllenCahn}. The equation (\ref{rr}) considered in this text is its steady-state version and the corresponding global minimisers represent the minimal-energy steady state solutions of the phase-field model in question. On the other hand, the Allen-Cahn equation gained its popularity in geometry via a conjecture of de Giorgi \cite{deGiorgi} that roughly asks whether the level sets of solutions of (\ref{rr}) are always minimal hyperplanes, at least in low-enough dimensional euclidean spaces. For the precise formulation of de Giorgi conjecture together with a very nice overview of the state of the art, we refer the interested reader to the works by Savin and del Pino \cite{Savin2, delpino, delpino2}. Suffice it to say, this question by de Giorgi launched an extremely fruitful and lively field of research dealing with the intricate connection between solutions of (\ref{rr}) and minimal hypersurfaces, first in the euclidean setting, and more recently also for other geometries. As explained in larger detail below, also in this setting minimal solutions play a special and important role. 

\subsubsection{Asymptotic behaviour of minimisers of the Allen-Cahn equation}

A partial answer to de Giorgi conjecture by Savin (\cite[Theorem 2.2]{Savin}) states that in $\R^d$, for $d\leq 7$, all level sets of local minimisers of (\ref{rr}) are hyperplanes. This in particular shows that global minimisers of the Allen-Cahn equation in that setting reduce to solutions of an ODE and that any global minimiser trapped between $c_0$ and $c_1$ forms a (particularly nice) heteroclinic connection from $c_0$ to $c_1$. Describing the asymptotics in the euclidean setting via the angular directions of rays emanating from the origin, the asymptotic behaviour of any such global minimiser can be thus described by a splitting of the $d$-sphere into two half spheres. 

The situation changes dramatically, when one considers the Allen-Cahn equation on a hyperbolic manifold  $\mathbb{H}^d$ of constant negative curvature $-1$. The following approach has been used by Pisante and Ponsiglione in \cite{italians}, by Birindelli and Mazzeo in \cite{birindelli}, and by Mazzeo and Saez in \cite{mazzeo}. Representing a hyperbolic manifold in the Poincar\'e ball model the asymptotic directions are given by points on the sphere that compactifies the ball, which may be again identified with the angles of the geodesic rays emanating from the origin. Similarly as in the Euclidian case, one may look for global minimisers with prescribed asymptotics by the half-spheres. By using the symmetry of $\mathbb{H}^d$ the PDE reduces to an ODE, which gives a heteroclinic connection between the two states. Observe that in the euclidean case a euclidean motion acting on a function with asymptotics defined by two half-spheres maps it to another function with asymptotics defined by two half-spheres. In the hyperbolic setting, in contrast, the group of isometries of $\mathbb{H}^d$ acts only quasi-conformally on the boundary and thus allows one to obtain minimisers with asymptotics given by any two open disjoint balls in the sphere, such that the closure of their union is the whole sphere. By the use of standard tools from the theory of elliptic PDEs (similar to those introduced in section \ref{var_prob}) one can then obtain for the hyperbolic space a similar statement to theorem A stated above. 

Our goal is to generalise this result to the case of an underlying metric space with a coarse hyperbolicity property and no inherent symmetries. In particular, the reduction to an ODE will not be possible and we shall need to use other arguments to show the existence of minimisers. 

\subsubsection{Gromov hyperbolic graphs and the isoperimetric inequality}

The coarsest concept of hyperbolicity, which generalises strictly negative sectional curvature in smooth Riemannian manifolds, is that of Gromov-hyperbolicity for geodesic metric spaces. Choosing a graph as the underlying metric space, provides a nice entry point into the analysis of nonlinear PDEs on Gromov hyperbolic spaces. Apart from providing an interesting and relevant setting by itself, its discrete nature simplifies the analysis as one does not need to deal with regularity issues. We are confident that the qualitative aspects of the theory obtained in this paper can be generalised to more general metric spaces and in particular to manifolds.

Gromov hyperbolicity is defined by the simple statement that all geodesic triangles need to be uniformly thin (see section \ref{geometric_graphs} for a brief introduction including relevant references). Typical examples of Gromov hyperbolic graphs are trees and more importantly Cayley graphs of fundamental groups of compact manifolds that allow a metric of pinched negative sectional curvature (e.g. the fundamental group of any closed orientable surface of genus larger than 1). A Gromov hyperbolic graph may be compactified by the so-called ``boundary at infinity'', which consists of equivalence classes of geodesic rays (see section \ref{geometric_graphs}). However, the boundary at infinity for a general Gromov hyperbolic graph can come also with a ``non-suitable'' topology from our point of view. This is connected to the fact that the definition of Gromov hyperbolicity allows some ``pathological'' examples, such as finite graphs or Cayley graphs of $\Z$ and $\Z\times (\Z/p\Z)$, which do not behave like manifolds of strictly negative sectional curvature. This forces us to put additional conditions on the graph if we want it to resemble a negatively curved manifold. 

The most important condition that we pose with respect to this problem comes in the form of an isoperimetric inequality (see section \ref{isoperimetric}), which estimates the number of vertices in an arbitrary subset of the graph, by the number of edges that form the boundary of this subset. For a $d$-dimensional manifold of negative sectional curvature bounded away from zero, a linear isoperimetric inequality that estimates the $(d-1)$-dimensional volume of the boundary of a set with respect to the $d$-dimensional volume of the set itself holds. In analogy, we assume in our setting a close-to linear isoperimetric inequality (see (\ref{IP})), which in particular bounds the size of balls of a particular radius from below. The isoperimetric inequality (\ref{IP}) is automatically satisfied for Cayley graphs of hyperbolic groups \cite{coulhon}, but not for general Gromov-hyperbolic graphs. An illustrative example is the following: by attaching to any edge of a Gromov-hyperbolic graph an infinite non-branching ray, this ray defines an isolated point in the boundary at infinity. On subsets of such a ray no isoperimetric inequality can hold and so an isoperimetric inequality excludes isolated points of the boundary at infinity. Apart from (\ref{IP}), some additional quite natural conditions on the graph are assumed in order to exclude pathological examples. For example, we assume the existence of a uniform bound on the number of edges at each vertex, which complementary to (\ref{IP}), gives a uniform bound on the size of a ball of a particular radius from above, thus generalising e.g. a bounded geometry condition for a negatively curved manifold. For all conditions and their precise formulations see section \ref{geometric_graphs}.

\subsubsection{Connection to the Dirichlet problem for harmonic functions}

The idea of generalising the theory for the Allen-Cahn equation on $\mathbb{H}^d$ described above to the setting of Gromov-hyperbolic spaces satisfying an isoperimetric inequality is specifically motivated by the so-called Dirichlet problem at infinity for harmonic functions. More precisely, the question is the following: given any regular-enough function $\varphi$ on the boundary at infinity, does there exists a harmonic function asymptotically converging towards the values of this $\varphi$ along geodesic rays of the corresponding equivalence classes (i.e. points at infinity)? In the context of simply connected Riemannian manifolds with sectional curvature pinched between two negative constants, this problem has been tackled among others by Choi \cite{choi}, Anderon \cite{anderson83} and Sullivan \cite{sullivan83}. In the context of Gromov hyperbolic graphs with similar conditions to the ones posed in this paper, such solutions have been obtained by Ancona \cite{ancona}, Coulhon \cite{coulhon} and Sullivan \cite{sullivan79}. Generalisations to other convex variational problems which have all the constants as solutions and to more general Gromov-hyperbolic metric spaces have been developed. For a state of the art result see the paper by Holopainen, Lang and V\"ah\"akangas \cite{holopainen_lang}. There is a difference between the isoperimetric inequality condition (\ref{IP}) from section \ref{isoperimetric} posed in this paper and the one used in the aforementioned works on harmonic functions. Namely, (\ref{IP}) is genuinely weaker than the strict linear isoperimetric inequality, which is usually posed or automatically satisfied in the literature that deals with the Dirichlet problem at infinity for the Laplace equation and its generalisations. More precisely, a strict linear isoperimetric inequality is satisfied on manifolds with negative-sectional curvature, while in other cases (\cite{ancona, coulhon, holopainen_lang}) the underlying metric space is assumed to satisfy a global Poincar\'e inequality, which is equivalent to it (see e.g. \cite{heinonen} about the equivalence). In particular, in our setting the results hold for all hyperbolic groups, while for harmonic functions this is not necessarily the case.

When generalising the Dirichlet problem at infinity for harmonic functions on a Gromov hyperbolic metric space $(\Gamma,d)$ to the Allen-Cahn equation, the following simple observation should be taken in account. Let $\varphi$ be a function from the boundary at infinity $\p \Gamma$ to $\R$, which is continuous at a point $\xi\in \p \Gamma$ and assume that the assymptotics of $x:\Gamma\to \R$ are given by $\varphi$. Then for any sequence of points $v_n\in \Gamma$, $n\in \N$, with $v_n\to \xi$ with $n\to \infty$, $\Delta_{v_n}x\to 0$ for $n\to \infty$. In particular, if $x$ is a solution of the Allen-Cahn equation (\ref{rr}), it must hold that $V'(\varphi(\xi))=0$. This shows that when one tries to solve the Dirichlet problem at infinity for the Allen-Cahn equation, one should focus on functions $\varphi:\p \Gamma\to \R$, with values only in the set of critical points of $V$ (which are for generic $V$ then locally constant). This is also one of the reasons, why the methods from the literature dealing with the Dirichlet problem at infinity for harmonic functions cannot be used. Namely, the two standard methods are either Perron's method of barriers (see \cite{choi, anderson83, holopainen_lang}), which uses the foliation of constant solutions in a critical way, or probabilistic methods (as in \cite{ancona, coulhon, sullivan79, sullivan83}), which use the non-linearity of the equation. Both methods are quite clearly not applicable to the problem discussed in this work. In view of the variational structure of the Allen-Cahn equation and the results presented above for the euclidean space and $\mathbb{H}^d$, it is thus natural to develop a variational approach for solving the Dirichlet problem at infinity for global minimisers, and more specifically, for those with asymptotics given from the set $\{c_0,c_1\}$, as stated in theorem A above.

\subsubsection{Variational approach and minimal boundaries}

A similar variational approach to the one used in this paper was developed by the author in \cite{ACAH} for solving the Allen-Cahn equation on a Cayley graph of a hyperbolic group. There, the equivalent statement to the existence theorem above is proved when (\ref{rr}) has the form $\rho\Delta_g(x)-V'(x_g)=0$ with asymptotically small $\rho$ (see \cite[Theorem 4.14]{ACAH}). In Theorem 3.7 of that paper it has been shown that for small enough $\rho$, all global minimisers converge to either $c_0$ or $c_1$, which gives additional motivation for studying the Dirichlet problem in the form presented in theorem A above. Even though a part of the construction of the proofs in this paper is similar to those in \cite{ACAH}, a couple of serious technical difficulties needed to be overcome to prove the statements in this text. The first and technically the most difficult problem was to move away from the perturbative case, that is, to drop the constant $\rho$ from the equation. The estimate contained in lemma \ref{ts_estimate} is the most important part of the solution. Apart from that, we were able to isolate the exact conditions on the metric space, which are needed for the methodology to work. Accordingly, we were able to generalise the underlying space from Cayley graphs to general hyperbolic graphs satisfying the conditions mentioned in the first section of the introduction. Note that hence no symmetry of the space is assumed. 

In \cite[section 5]{ACAH} it is furthermore shown that by taking $\rho$ to zero, the sets where the global minimisers transit from $c_0$ to $c_1$, converge to a solution of an asymptotic Plateau problem. Roughly speaking, this shows the existence of sets with prescribed asymptotics that have minimal boundaries with respect to compact variations (see \cite[section 5]{ACAH} for precise statements). The approach is inspired by the famous paper of Modica and Mortola \cite{modica}, where they construct minimal hypersurfaces via $\Gamma$-convergence from solutions of the Allen-Cahn equation in the euclidean space. Applying the methods from \cite{ACAH} to hyperbolic graphs satisfying conditions introduced in this paper, solutions to the asymptotic Plateau problem can be constructed in an analogous way. We refrain, however, from presenting the complete argument in this text, since the non-perturbative setting (i.e. no small constant $\rho$) developed in this paper is not optimal to deal with that problem directly.% and on the other hand most of the proofs from \cite{ACAH}, can be directly applied to the setting considered in this paper. 

We should here also mention that other existence results about minimisers for more general nonlinear variational elliptic operators on quite general Cayley graphs of groups have been discussed in \cite{llave-lattices,candel-llave} in the context of Aubry-Mather theory. However, in those works minimisers on the abelianisation of the group are discussed, which are not necessarily global minimisers of the group itself. Moreover, non-constant minimisers constructed in that way are highly irregular at the boundary at infinity.

\subsection{Outline}
The rest of this text is structured as follows. In section \ref{geometric_graphs} we pose the precise conditions on the graph. We give a brief overview of Gromov-hyperbolic graphs, including the compactification by the boundary at infinity and the visual metric. We discuss an upper bound on the growth of the graphs in connection to the Assouad dimension of the boundary at infinity, the proof of which is given in the appendix. Then we introduce the isoperimetric inequality and define geometric objects called shadows and cones, which generate the so-called ``cone topology'', equivalent to the one induced by the visual metric. Some geometric properties of the boundaries of cones are furthermore investigated.

In section \ref{var_prob} we discuss the variational structure of the Allen-Cahn equation and gather some standard results from elliptic PDE theory, such as the existence of minimal solutions trapped between $c_0$ and $c_1$ on any compact set if the boundary conditions are from $[c_0,c_1]$. By taking an exhaustion of the graph by balls with increasing radii, such solutions then converge to a global minimiser and our goal is to control its asymptotics. The first step are two technical variational lemmas, which follow from energy estimates of a minimal solution and lie at the heart of the proofs. 

In section \ref{dp}, we use the isopertimetric inequality together with the geometric properties of the boundaries of cones and the variational lemmas, to show that asymptotics of a global minimiser can be controlled. We use the intuitive fact that the set of vertices where a global minimiser passes from one state to the other (the transition set) needs to be small, as the contribution of $V$ to the action is on this set large. We show that in a neighbourhood of $D_1$ (the set where the asymptotic behaviour is given by $c_1$) this transition set extends uniformly towards the base vertex, where the distance to the base vertex depends on the largest radius of a ball at infinity which is fully contained in $D_0$.

\subsection{Acknowledgement }
I would like to thank Prof.~V.~Bangert for the helpful conversations and for his valuable comments.\\ 

\noindent
The final publication is available at Springer via \\
http://dx.doi.org/10.1007/s00526-016-1100-x

%T small step of investigating Moser-type elliptic PDEs on hyperbolic manifolds and similar.

\section{Hyperbolic graphs}\label{geometric_graphs}

In this section we gather some basic facts about hyperbolic graphs and pose the necessary conditions. Let $\Gamma$ denote an infinite connected graph with the vertex set $G$ that is uniformly locally finite: \begin{equation*}\text{{\it there exists an $S\in \N$, such that every vertex $v\in G$ joins to at most $S$ edges.}}\end{equation*} By setting the length of the edges to one $\Gamma$ becomes a metric space, where the metric $d(\cdot,\cdot)$ is defined as follows: for every two vertices $g,\tilde g\in G$, $d(g,\tilde g)$ is the least length of a path in $\Gamma$ connecting $g$ and $\tilde g$. A curve $\gamma_{g,\tilde g}$ from $g$ to $\tilde g$ of length $l(\gamma_{g,\tilde g})=d(g,\tilde g)$ is called a geodesic and needs not to be unique. It follows that the space $(\Gamma, d)$ is an unbounded complete geodesic metric space, i.e. geodesics between any two points exist. 

We assume that there is a constant $\delta>0$, such that $(\Gamma, d)$ is a $\delta$-hyperbolic metric space in the sense of Gromov (as introduced in \cite{gromov1,gromov2}). This means that every geodesic triangle,(i.e.~a set of three points $g_1,g_2,g_3\in \Gamma$ together with any three geodesics $\gamma_{g_1,g_2},\gamma_{g_2,g_3},\gamma_{g_3,g_1}\subset \Gamma$ connecting them) is $\delta$-slim (i.e.~$\gamma_{g_1,g_2}$ is in the $\delta$-neighbourhood of $\gamma_{g_2,g_3}\cup \gamma_{g_3,g_1}$). There exists an extensive literature on the subject of Gromov hyperbolic spaces. For basic definitions and results, we refer the reader to \cite{harpe,notes1,notes2}. Typical examples of $\delta$-hyperbolic graphs are Cayley graphs of hyperbolic groups, such as fundamental groups of manifolds with strictly negative sectional curvature and free groups. Furthermore, hyperbolic graphs appear in network theory and mathematical biology.

\subsection{Compactifying the graph}\label{compactifying}

We shall briefly discuss the compactification of locally uniformly finite Gromov hyperbolic graphs in this section. For extensive overviews on this and related subjects we refer the reader to \cite{gromov1,kapovich,calegari}. We shall provide more specific references along the way.

A geodesic ray is given by an isometry $\gamma:[0,\infty)\to \Gamma$, where $[0,\infty)$ is equipped with the standard metric. On $\delta$-hyperbolic spaces, one may define an equivalence relation on the set of geodesic rays, by $\gamma_1\sim \gamma_2$, if there exists a $C\in \R$ such that $d(\gamma_1(t),\gamma_2(t))\leq C$ for all $t\in \R_+$. One then defines the boundary at infinity $\p \Gamma$ as the set of such equivalence classes of rays. Since $\Gamma$ is assumed to be unbounded, $\p \Gamma\neq \varnothing$.
Defining $\overline{\Gamma}:=\Gamma\cup\p\Gamma$, we may extend any ray $\gamma\in \xi\in \p\Gamma$ to $\gamma:[0,\infty]\to \overline{\Gamma}$, by defining $\gamma(\infty)=\xi$. It then holds that for every $g\in \Gamma$ and $\xi \in \p \Gamma$ there exists a geodesic ray $\gamma_{g,\xi}\subset \overline{\Gamma}$, such that $\gamma_{g,\xi}(0)=g$ and $\gamma_{g,\xi}(\infty)=\xi$. Moreover, for any two points $\xi, \mu \in \p \Gamma$ there exists an (infinite) geodesic $\gamma_{\xi,\mu}\subset \overline{\Gamma}$, connecting these two points at infinity, i.e. the ray $\gamma_{\xi,\mu}|_{[0,-\infty)}$ belongs to the equivalence class $\xi$ and $\gamma_{\xi,\mu}|_{[0,\infty)}$ to $\mu$. With these definitions, geodesic triangles in $\overline{\Gamma}$ are also $\tilde \delta$ slim for a uniform constant $\tilde \delta$. 

For the remainder of this text, let us assume that there exists a base-vertex $v \in G$ and a constant $\bar \delta>0$, such that $\Gamma$ is {\it$\bar \delta$-visual} from $v$. More precisely, let us define for every $\xi\in \p \Gamma$ the set $\xi_{v}\subset G$ by \begin{equation}\label{xiv}\xi_{v}:=\bigcup_{\gamma=\gamma_{v,\xi}} (\gamma\cap G).\end{equation}
We assume then that for every vertex $g\in G$ there exists a $\xi\in \p \Gamma$, such that \begin{equation}\label{visual} d(g,\xi_{v})\leq \bar \delta . \end{equation} %(maybe doesn't need a citation, otherwise c.f. p.20 in \cite {ancona88} and \cite{BonkSchramm}). 
For simplicity, we redefine $\delta$ as the maximum of $\tilde \delta, \bar \delta$ and $\delta$. 

Denote for every $g\in G$ its distance to the base vertex $v$ by $|g|:=d(v,g)$. The so-called {\it visual metric} makes $\overline{\Gamma}$ into a bounded geodesic metric space. It is defined for every $y,\tilde y\in \Gamma$ by $$d_\varepsilon(y,\tilde y):= \inf_{p_{y,\tilde y}}\int_0^{l(p_{y,\tilde y})}e^{-\varepsilon \cdot d(v, p_{y,\tilde y}(s))}ds \ ,$$ where $p_{y,\tilde y}\subset \Gamma$ is a path from $y$ to $\tilde y$. Let $\xi,\mu\in \p \Gamma$ and let $y\in \gamma_1$ where $[\gamma_1]= \xi$ and $\tilde y\in \gamma_2$ where $[\gamma_2]= \mu.$ It follows from $\delta$-hyperbolicity that there exists an $\varepsilon_0>0$, such that for all $0<\varepsilon<\varepsilon_0$, the limit of $d_\varepsilon(y,\tilde y)$ for $|y|,|\tilde y|\to \infty $ is uniformly bounded and that it depends only on the equivalence classes $\xi$ and $\mu$. This gives us for every $0<\varepsilon<\varepsilon_0$ a metric on $\overline{\Gamma}$, for which there exists a constant $\lambda>0$, such that for all $\xi,\mu\in \p \Gamma$ and every infinite geodesic $\gamma_{\xi,\mu}$\begin{equation}\label{vm_estimate}\lambda^{-1}e^{-\varepsilon d(o, \gamma_{\xi, \mu})}\leq  d_\varepsilon(\xi, \mu)\leq \lambda e^{-\varepsilon d(o, \gamma_{\xi, \mu})}. \end{equation} Let us fix such an $\varepsilon\in (0,\varepsilon_0)$ for the rest of this paper.

For every $n\in \N$ we denote metric balls in $G$ by 
$$\mathcal{B}_n:=\{g\in G \ | \ |g|\leq n\} \  \text{ and }   \ \mathcal{B}_n(g_0):=\{ g\in G \ | \ d(g_0, g)\leq n\},$$ 
and we denote balls of radius $r>0$ in $\p \Gamma$, or in $G\cup \p \Gamma$ with respect to the visual metric by \begin{equation*}\begin{aligned}
&B^\varepsilon_r(\xi_0):=\{\xi \in \p \Gamma \ | \ d_\varepsilon(\xi_0,\xi)\leq r \}\subset \p \Gamma, \\
&\mathcal{B}^\varepsilon_r(y_0):=\{y \in G \cup \p \Gamma \ | \ d_\varepsilon(y_0,y)\leq r \}\subset G\cup \p \Gamma,
\end{aligned}\end{equation*}  where $\xi_0\in \p \Gamma$ and $y_0 \in G \cup \p \Gamma$.
For $\mathcal{D}\subset G\subset \Gamma$ we denote by $\# \mathcal{D}$ the number of vertices contained in $\mathcal{D}$. Then $$\# \mathcal{B}_n\leq S^n+S^{n-1}+\dots+S+1=\frac{S^{n+1}-1}{S-1}$$ and with the visual metric, we can obviously find constants $D>0$, $C_D>0$, such that \begin{equation}\label{gob}\# \mathcal{B}_n\leq C_De^{\varepsilon nD}.\end{equation}

A natural question arises of what is the optimal constant $D>0$ for a specific graph $\Gamma$ satisfying the given conditions. This is in our setting especially important, since $D$ also comes up in the isoperimetric condition (\ref{IP}) below. A more involved estimate, stated in the proposition below, can be obtained in terms of the Assouad dimension $d_A(\p \Gamma)$ of the metric space $(\p \Gamma, d^\varepsilon)$, which is in our case finite (see \cite{BonkSchramm}) and might provide a better estimate in specific cases. A more thorough discussion about the Assouad dimension and the proof of the following statement are contained in the appendix A.

\begin{proposition}[Growth of balls]\label{goc}
For every $D>d_A(\p \Gamma)$ there exists a constant $C_D>0$, such that for every $n\in \N$, $$\# \mathcal{B}_n\leq C_De^{\varepsilon D n} .$$
\end{proposition}

%It is easy to see that $\overline \Gamma$ is with the visual metric bounded and complete (see \cite{BonkSchramm}). Because $\Gamma$ is uniformly locally bounded, it is even totally bounded so it is compact (maybe citation or explanation). In fact, local uniform boundedness has as a consequence that $\p\Gamma$ has finite Assouad dimension, which implies that it is a doubling metric space \cite{BonkSchramm} and because it is complete it even carries a doubling measure (see \cite{heinonen}). The Assouad dimension can be in turn used to obtain a better constant $D$ in the estimate \ref{gob} (will you do this?). 

\subsection{Boundaries of sets and the isoperimetric profile}\label{isoperimetric}

Let us define for a set $\mathcal{B}\subset G$ its ``outer'' set by $\mathcal{B}^{\out}:=\{g \in G \ | d(g,\mathcal{B})\leq 1\}$ and its ``inner'' set $\mathcal{B}^{\inn}$ by $(\mathcal{B}^{\inn})^{\out}=\mathcal{B}.$ Furthermore, we define the outer and the inner boundaries of $\mathcal{B}$ by $\p^{\out} \mathcal{B}:=\mathcal{B}^{\out} \backslash \mathcal{B}$ and $\p^{\inn} \mathcal{B}:=\mathcal{B} \backslash \mathcal{B}^{\inn}$ and the ``full'' boundary by $\p^{\f} \mathcal{B}:=\p^{\out} \mathcal{B}\cup\p^{\inn} \mathcal{B}$. It easily follows that \begin{equation}\label{boundary}\p^{\out}(\mathcal{B}\cap \mathcal{D})=(\p^{\out} \mathcal{B}\cap  \mathcal{D}^{\out})\cup (\mathcal{B}^{\out}\cap \p^{\out} \mathcal{D}) \end{equation} for any two set $\mathcal{B},\mathcal{D}\subset G$.

Recall that the concept of Gromov hyperbolicity allows graphs such as the Cayley graph of $\Z$ or $\Z_n\times \Z$. However, a fundamental difference in the classes of solutions to PDEs on geometric spaces appears between spaces of polynomial vs.~exponential growth. We shall thus need to put an extra condition on our space, which will ensure the minimal required local growth properties of $G$. It comes in the form of an isoperimetric inequality.

Let $D>0$ be as in (\ref{gob}) and assume that there exists a constant $C_0\geq 1$, such that for every $B\subset G\subset \Gamma$ the following {\it isoperimetric inequality} holds:
\begin{equation}\label{IP}C_0 \# (\p^{\out} B) \geq (\# B)^{\left(\frac{4D}{4D+1}\right)} .\tag{IP}\end{equation}

\begin{remark}
As explained in the paragraph containing (\ref{gob}), Gromov hyperbolicity together with local uniform boundedness of $\Gamma$ give an upper bound on its growth. The (\ref{IP}) condition, on the other hand ensures a lower growth bound on (all subsets of) $\Gamma$. In particular, it ensures that there are no isolated (equivalence classes of) rays  in $\p \Gamma$ and is equivalent to a global Poincar\'e inequality (see \cite{heinonen}).

Note that (\ref{IP}) holds for hyperbolic groups, for which there exist constants $\alpha>0$ and $\tilde C>0$, such that for every $B\subset G\subset \Gamma$
$$\tilde C \# (\p^{\out} B)\geq \frac{\# B}{\log(\alpha(\# B))}.$$ This statement can be found e.g. in \cite{coulhon}.

%Note that we may assume for technical reasons that $C_0\geq 1$. 
\end{remark}

\subsection{Shadows and cones} 

We adapt the definition of a shadow from the setting of hyperbolic groups (due to Coornaert and Sullivan \cite{coornaert93, sullivan79}).

\begin{definition}
For every $g\in G$ we define the shadow of $g$ by $$S(g):=\{\xi\in \p \Gamma \ | \ d(g,\xi_v)\leq \delta\}\subset \p \Gamma.$$ 
\end{definition}

Clearly, for every $\xi\in \p \Gamma$ there exists a $g\in G$ with $\xi \in S(g)$ and since $\Gamma$ is visual, $S(g)\neq \varnothing$ for any $g\in G$. 

\begin{remark}
If $\Gamma$ is a Cayley graph of a hyperbolic group, the definition of a shadow in \cite{coornaert93, sullivan79} requires a much larger distance from the point $g$. There, shadows can be used together with the automaton structure and the quasi-conformal action of the group on its boundary to obtain useful concepts of measure and dimension of the boundary at infinity (see \cite{notes2,ACAH}). One step in the proofs that can be applied also to a graph like $\Gamma$ is the following proposition, which states that shadows are ``almost round''. The proof is analogous to \cite[Proposition 2.3]{ACAH}.
\end{remark}

\begin{proposition}\label{shadow}
There exists a constant $C_1>0$ such that for every $\xi\in \p \Gamma$, $r>0$ and $g_1,g_2\in \xi_{v} \subset G$, $$\begin{aligned}B_r^\varepsilon(\xi) \subset S(g_1)& \ \text{ if } \ C_1 e^{-\varepsilon|g_1|}\geq r \text{ and }\\ S(g_2)\subset B_r^\varepsilon(\xi)& \ \text{ if } \ C_1^{-1} e^{-\varepsilon|g_2|}\leq r.\end{aligned}$$
Moreover, there exists a constant $C_2>0$ such that for every $g\in G$ and $\xi\in S(g)$, $$S(g)\subset B_r^\varepsilon(\xi)\ \text{ if } \ r\geq C_2 e^{-\varepsilon |g|}.$$ 
\end{proposition}

\begin{proof}
Let $\xi,\eta\in \p \Gamma$, $r>0$ and let $d_\varepsilon(\xi,\eta)<r$. Denote by $\gamma_{\xi,\eta}$ a geodesic that connects $\xi$ to $\eta$ in $\Gamma$. Then by (\ref{vm_estimate})  $$\lambda^{-1} e^{-\varepsilon d(\gamma_{\xi,\eta},v)}\leq d_\varepsilon(\xi,\eta)\leq r.$$ Let $g_1\in \xi_{v}$ with $ \lambda^{-1} e^{-\varepsilon(|g_1|+ \delta)}\geq r$. Since geodesic triangles in $\overline{\Gamma}$ are $\delta$-slim, it follows that $d(g_1,\eta_{v})\leq \delta$ and so $\eta\in S(g_1)$ which proves the first inclusion.

The second inclusion follows via a similar proof. 

Let $\xi,\eta \in S(g)$ and choose points $g_\xi \in \mathcal{B}_{ \delta}(g)\cap \xi_{v}$ and $g_\eta\in \mathcal{B}_{ \delta}(g)\cap \eta_{v}$. By the triangle inequality $$d_\varepsilon(\xi,\eta)\leq d_\varepsilon(\xi,g_\xi)+ d_\varepsilon(g_\xi, g_\eta)+d_\varepsilon(g_\eta,\eta).$$ Since $g_\xi, g_\eta\in \mathcal{B}_{ \delta}(g)$, $d(g_\xi, g_\eta)\leq 2\delta$ and it follows that $d(v, \gamma_{g_\xi, g_\eta})\leq |g|-2 \delta$ and by (\ref{vm_estimate}) that $$d_\varepsilon(\xi,\eta)\leq \lambda e^{-\varepsilon |g_\xi|}+\lambda e^{-\varepsilon (|g|-2\delta)}+ \lambda e^{-\varepsilon |g_\eta|}\leq 3\lambda e^{2\delta} e^{-\varepsilon|g|}.$$
\end{proof}

%\subsubsection{A homogeneity condition}
%
%The following condition ensures that the graph uniformly grows, even on subsets.
%
%Assume that there exists a measure $\nu$ supported on $\p \Gamma$ and positive real numbers $C_3,C_4>0$ and $D_1\geq D_2>0$, such that for all $\xi\in \p \Gamma$ and $r>0$:
%
%\begin{equation}\label{condition_measure}
%C_3r^{D_1}\leq \nu(B^\varepsilon_r(\xi))\leq C_4r^{D_2}.
%\end{equation}
%
%Note that since $\p \Gamma$ is bounded this is a condition on small balls (this is why $D_1\geq D_2$). It is not clear how and if the conditions (\ref{IP}) and (\ref{condition_measure}) are connected, but it seems plausible that the isoperimteric condition (\ref{IP}) might imply the condition on the measure at infinity (\ref{condition_measure}). In any case it is easy to see that both conditions exclude isolated points at infinity.
%
%%Assume that there exists a measure $\nu$ supported on $\p \Gamma$ and positive real numbers $C,D>0$, such that for all $\xi\in \p \Gamma$ and $r>0$:
%%
%%\begin{equation}\label{condition}
%%C^{-1}r^D\leq \nu(B^\varepsilon_r(\xi))\leq Cr^D.
%%\end{equation} This is called Ahlfors-regular, or just $D$-regular.
%
%\begin{remark}
%In case of groups, $\nu$ is the Patterson-Sullivan measure, $h$ is the entropy of the group and $D_2=D_1=h/\varepsilon$ (see \cite{calegary,AHAC}).
%\end{remark}

%It follows from \cite[section 9]{BonkSchramm}, that the Assouad dimension of $\p \Gamma$ is finite and furthermore that $\p \Gamma$ is a doubling metric space. There is a natural measure, which is doubling, so this can be seen as a condition on that measure. 

Using shadows, we can now define a topology on $G$, as in \cite{ACAH}.

\begin{definition}\label{cone}
For a set $U\subset \p \Gamma$ define the $U$-{\it cone} by $$\mathcal{C}_{U}:=\{ g \in G \ |\ S(g) \subset U\}.$$ 
Furthermore, we define the {\it cone topology} on $G$ as the topology generated by balls $\mathcal{B}_n(g_0)$ and truncated cones $C_{B^\varepsilon_r(\xi_0)}\backslash \mathcal{B}_n$. 
\end{definition}

In its essence, the cone topology defined above is similar to the cone topology on manifolds with pinched negative sectional curvature (see e.g. \cite{choi}). The following statement is essentially proved in \cite[lemma 2.6]{ACAH}, even though shadows are defined in a slightly different way there. The proof, however, is only based on the fact that shadows are almost-round as stated in proposition \ref{shadow} (c.f. \cite[proposition 2.1]{ACAH}).

\begin{lemma}\label{topology}The cone-topology from definition \ref{cone} is equivalent to the topology induced by the visual metric. 
\end{lemma}

For the remainder of this section, we shall focus on $U$-cones, where $U$ is a metric ball at infinity.

In definition \ref{cone} we associate to every $g\in \mathcal{C}_{B^\varepsilon_r(\xi_0)}$ the set of points $S(g)\subset B^\varepsilon_r(\xi_0)$. On the other hand, we shall find it useful to associate to every $\xi \in B^\varepsilon_r(\xi_0)$ the set  $$\mathcal{U}_\xi:=\{g \in G \ | \ d(g,\xi_{v})\leq \delta\}.$$ 

\begin{lemma}\label{estimate_cone}
Let $C_1$ be as in proposition \ref{shadow}. If $g\in \mathcal{U}_\xi$ for some $\xi\in B^\varepsilon_r(\xi_0)$ with $d_\varepsilon(\xi, \xi_0)< r-C_1 e^{-\varepsilon|g|}$, then $g\in \mathcal{C}_{B^\varepsilon_r(\xi_0)}$.
\end{lemma}

\begin{proof}
By proposition \ref{shadow} $S(g)\subset B_{C_1 e^{-\varepsilon |g|}}^\varepsilon(\xi)$ for every $\xi \in S(g)$. So if $g\in \mathcal{U}_\xi$ for a $\xi$ with $d_\varepsilon (\xi, \xi_0)< r- C_1 e^{-\varepsilon|g|}$, then $$S(g)\subset B_{C_1 e^{-\varepsilon |g|}}^\varepsilon(\xi) \subset B^\varepsilon_{r}(\xi_0)$$ and so $g\in \mathcal{C}_{B^\varepsilon_r(\xi_0)}$. 
\end{proof}

%\subsubsection{Boundaries of cones} \label{sep_sets}

To conclude this section we prove a somewhat technical statement about estimating the boundaries of cones with annuli at infinity. Define an annulus by $$ A_r^t(\xi_0):=\overline{(B_{r+t}(\xi_0)\backslash B_{r-t}(\xi_0))}\subset \p \Gamma$$ and denote for a set $\mathcal{D}\subset G$, the $n$-times iterated inner set $(\dots((\mathcal{D})^{\inn})^{\inn}\dots)^{\inn}$ by $\mathcal{D}^{(n\inn)}$ and similarly the $n$-times iterated outer set $\mathcal{D}^{(n\out)}$.  %Let $r>0$ be small enough, so that $\p \Gamma\backslash {B^\varepsilon_r(\xi_0)}\neq \varnothing$. 

\begin{lemma}\label{separating_sets}
Define for every $n\in \N$ the number $t_n:= \max\{4\varepsilon^{-1}e^{\varepsilon4},C_2\}e^{-\varepsilon n}$ and for every $r\geq 0$ the set $$\mathcal{A}_{r,t_n}:=\{g\in \mathcal{U}_\xi \backslash \mathcal{B}_n\ | \ \xi\in A_{r+2t_n}^{t_n}(\xi_0)\}.$$ Then $$\p^{\f} (\mathcal{C}_{B^\varepsilon_{r+2t_n}(\xi_0)})^{(3\out)}\subset \mathcal{A}_{r,t_n}\cup \mathcal{B}_n\ \text{ and }\ \p^{\out} \mathcal{C}_{B^\varepsilon_{r+2t_n}(\xi_0)}\subset \mathcal{A}_{r,t_n}\cup \mathcal{B}_n.$$ Moreover, $(\mathcal{A}_{r,t_n}\cap \mathcal{A}_{r+4t_n,t_n})\backslash \mathcal{B}_n=\varnothing$.
\end{lemma}

\begin{proof}
Let $g\in (\p^{\f} (\mathcal{C}_{B^\varepsilon_{r+2t_n}(\xi_0)})^{(3\out)}\backslash \mathcal{B}_n) \cup (\p^{\out} \mathcal{C}_{B^\varepsilon_{r+2t_n}(\xi_0)}\backslash \mathcal{B}_n)$. Then there exist a $\xi\in  \Gamma\backslash B_{r+2t_n}^{\varepsilon}(\xi_0)$ with $g\in \mathcal{U}_\xi\backslash \mathcal{B}_{n}$ and a vertex $\tilde g\in \mathcal{C}_{B^\varepsilon_{r+2t_n}(\xi_0)}\backslash \mathcal{B}_{n-4}$ such that $d(g,\tilde g)\leq4$.
It holds that $S(\tilde g) \subset B_{r+2t_n}^{\varepsilon}(\xi_0)$ and it follows by the definition of the visual metric and a short calculation for every $\eta\in S(\tilde g)$ that $$d_\varepsilon(\xi,\eta)\leq d_\varepsilon(\xi,g)+d_\varepsilon(g,\tilde g)+d_\varepsilon(\tilde g,\eta)\leq 4\varepsilon^{-1}e^{-\varepsilon(n-4)}\leq t_n,$$ so $\xi\in A_{r+2t_n}^{t_n}$.

To show the disjointness, let $g\in \mathcal{U}_\xi\backslash \mathcal{B}_n$ for a ray $\xi \in A_{r+2t_n}^{t_n}$. Then it holds for all rays $\eta$ with $g\in \mathcal{U}_\eta\backslash \mathcal{B}_n$ that $\eta\in S(g)$ and by proposition \ref{shadow} it follows that $d_\varepsilon(\xi,\eta)\leq C_2 e^{-\varepsilon n}< t_n$. Because of the strict inequality $\eta\in \mathring{A}_{r+2t_n}^{2t_n}$ and so $g\notin \mathcal{A}_{r+4t_n,t_n}$.
\end{proof}

\section{The variational problem}\label{var_prob}

As explained in the introduction, we are interested in solutions $x:G\to\R$ of the discrete Allen-Cahn equation (\ref{rr}), given by $$ \Delta_g(x)-V'(x_g)=0 , \text{ for all } \ g\in G ,$$ where $V:\R\to \R$ is a $C^2$ double-well potential. More precisely, we assume that $V$ is a $C^2$ function with two non-degenerate consecutive minima $c_0,c_1\in \R$. We assume without loss of generality, by adding a constant to $V$ if necessary, that $V(c_0)=V(c_1)=0$.

\subsection{Minimal solutions}

For any finite (compact) set $\mathcal{B} \subset G$ and function $x:G\to \R$, the ``truncated action functional'' is defined similarly as in (\ref{fe}) by \begin{equation}\label{vp}W_\mathcal{B}(x):=\sum_{g\in \mathcal{B}}\left(\frac{1}{4}|\nabla_g(x)|^2+V(x_g)\right).\end{equation} The function $W_\mathcal{B}(x)$ is then a function of variables $x_g$ where $g\in \mathcal{B}^{\out}$ and $x$ solves (\ref{rr}) for all $g\in \mathcal{B}^{\inn}$, if for every perturbation $y$ supported on $\mathcal{B}^{\inn}$, $$\left.\frac{d}{ds}\right|_{s=0} W_B(x+sy)=0.$$ %More precisely, if $v=\rho_g$, it holds that  $$\left.\frac{d}{ds}\right|_{s=0} W^\rho_B(x+sv)=\sum _{s\in S}\Delta_g(x)-V'(x_g).$$
This motivates the following definition.

\begin{definition}\label{def_gm}For any set $\mathcal{B}\subset G$, a function $x:G \to \R$ is called {\it a minimiser on $\mathcal{B}$}, if $$ W_{\mathcal{B}^{\out}}(x+y)-W_{\mathcal{B}^{\out}}(x)\geq 0$$ for all $y:G \to \R$ with compact support in $\mathcal{B}$. 

This definition makes sense also for infinite sets $\mathcal{B}$, because the support of $y$ is compact, since one may evaluate the difference above by truncating the actions to a bounded set containing the support of $y$ in its interior.
The function $x$ is called a {\it global minimiser}, if $x$ is a minimiser on $G$.
\end{definition}

Observe that with this definition, a minimiser $x$ on a compact set $\mathcal{B}$ is a solution to the Dirichlet problem given by (\ref{rr}) on $\mathcal{B}$, with given boundary values $x|_{ (\mathcal{B}^{\out})^{\out}\backslash \mathcal{B}}$. Such minimisers exist:

\begin{lemma}\label{existence}Let $\mathcal{B}\subset G$ be a compact set and let $f:(\mathcal{B}^{\out})^{\out}\backslash \mathcal{B} \to \R$ be given. Then there exists a minimiser $x$ of (\ref{vp}) on $\mathcal{B}$, with boundary values given by $f$, that is $x|_{(\mathcal{B}^{\out})^{\out}\backslash \mathcal{B}}=f$.
\end{lemma}

\begin{proof}
Since $V$ is non-negative $\tilde m:=\inf_x W_{\mathcal{B}^{\out}}(x)\geq0$ holds and because $\mathcal{B}$ is finite, $\tilde m<\infty$. Let $x^n$ be a sequence such that $x^n|_{(\mathcal{B}^{\out})^{\out}\backslash \mathcal{B}}=f$ and $\lim_n W_{\mathcal{B}^{\out}}(x^n)=\tilde m$. Since $W_{\mathcal{B}^{\out}}(x^n) \to \infty$ if for any $g\in \mathcal{B}$ and $s\in S$, $(x^n_{gs}-x^n_g)\to \pm \infty$, it follows that there exists a constant $K>0$ and an $N\in \N$, so that for all $n\geq N$ and all $g\in \mathcal{B}$, $|x^n_g-x^n_{gs}|\leq K$. Since $\mathcal{B}$ is finite, by local uniform finiteness of $\Gamma$ also $(\mathcal{B}^{\out})^{\out}$ is, and so for all $g\in \mathcal{B}$ and $n\geq N$, $|x^n_g|\leq C:=K(\max_{g}f(g))\#\mathcal{B}$. Hence, $x^n$ are contained in a compact set and thus due to continuity of $W_{\mathcal{B}^{\out}}$ converge to a minimiser $x^\infty$ satisfying the boundary conditions given by $f$.
\end{proof}

\begin{remark}\label{xc01}It is easy to see that the constant function $x^{c_0}\equiv c_0$ and $x^{c_1}\equiv c_1$ are global minimisers of (\ref{vp}).
\end{remark}

Obviously, all global minimisers are solutions and it is clear that one may construct global minimisers as limits of minimisers on compact domains that exhaust $G$. More precisely, if $\mathcal{B}^n$ is a sequence of compact subsets exhausting $G$ and $x^n$ a sequence of minimisers on $\mathcal{B}^n$ which converges to a function $x^\infty$, then $x^\infty$ is a global minimiser.

The following two lemmas are standard for elliptic difference operators. For proofs, see \cite{ACAH},  \cite{candel-llave} or \cite{llave-lattices}.

\begin{lemma}\label{mm_lemma}
For functions $x,y$, their point-wise minimum and maximum $m:=\min\{x,y\}$ and $M:=\max\{x,y\}$, and for any compact domain $\mathcal{B}\subset G$ it holds that $$W_{\mathcal{B}}(x)+W_{\mathcal{B}}(y)\geq W_{\mathcal{B}}(M)+W_{\mathcal{B}}(m).$$
\end{lemma}

\begin{lemma}\label{comparison}
Let $x,y$ be two minimisers on $\mathcal{B}$ with $x_g\leq y_g$ for all $g\in (\mathcal{B}^{\out})^{\out}\backslash \mathcal{B}$. Then either $x_g<y_g$ for all $g\in \mathcal{B}$, or $x|_{\mathcal{B}}\equiv y|_{\mathcal{B}}$.

In particular, any two global minimisers $x,y$ with $x_g\leq y_g$ for all $g\in G$ and such that $x\neq y$, are totally ordered: $x_g<y_g$ for all $g\in G$.
\end{lemma}

\begin{corollary}\label{trapping}
Let $f:(\mathcal{B}^{\out})^{\out}\backslash \mathcal{B} \to [c_0,c_1]$ be given. Then it holds for every minimiser $x$ on $\mathcal{B}$ satisfying the boundary conditions given by $f$ that $c_0\leq x_g\leq c_1$ for all $g\in(\mathcal{B}^{\out})^{\out}$. 
\end{corollary}

\subsection{Minimal Dirichlet problem at infinity}

We wish to construct global minimisers that solve the minimal Dirichlet problem at infinity as given in the following definition. 

\begin{definition}\label{dp_def}
Endow $G\cup \p \Gamma$ with the topology induced by the visual metric (which is by lemma \ref{topology} equivalent to the cone topology from definition \ref{cone}). Given two sets $D_0, D_1\subset \p \Gamma$ with $\overline{\mathring{D}}_0=D_0$ and $\overline{\mathring{D}}_1=D_1$, such that $\mathring{D_0}\cap \mathring{D_1}=\varnothing$ and $\overline{D_0}\cup \overline{D_1}=\p \Gamma$,
we say that a global minimiser $x$ solves the minimal Dirichlet problem at infinity for the equation (\ref{vp}) on $D_0$ and $D_1$, if the following holds.

For every $\xi \in \mathring{D}_j$ with $j\in \{0,1\}$ and for any $\epsilon>0$ there exists an open neighbourhood $\mathcal{O}_\xi\subset G\cup \p \Gamma$ of $\xi$, such that for all $g\in G\cap\mathcal{O}_\xi$, $|x_g-c_j|<\epsilon$.
\end{definition}

In the following proposition, we shall construct a global minimiser as a limit of local minimisers $x^N$ on balls $\mathcal{B}_N$ with growing radii $N\in \N$, using the concept of the cone from definition \ref{cone}. The difficult part, which the rest of the paper is then devoted to, deals with the asymptotics of such a global minimiser. 

From here on we shall focus on the asymptotics of $x^N$ in $D_1$, i.e. we aim to show that if $\xi \in \mathring{D}_1$ and $x^N$ a sequence of such local minimisers, then for every $\varepsilon>0$ we may find a small neighbourhood of $\xi$ in $G\cup \p \Gamma$, such that for all $g$ in this neighbourhood and for all large enough $N$, $|x^N_g-c_1|<\varepsilon$. To get the final statement we note that modulo some trivial reformulations of the conditions, the roles of $c_0$ and $c_1$ may be swapped on every stage of the discussion below.

\begin{proposition}\label{local_minimisers}
Let $c_0< c_1$ be the two distinct absolute minima of $V$ and let $D_0, D_1\subset \p \Gamma$ be as in definition \ref{dp_def}.
Define a function $\tilde x$ by $$\displaystyle \tilde x_g:=\begin{cases}c_1 & \text{ if } g\in \mathcal{C}_{\mathring{D}_1} ,\\ c_0 & \text{ else.} \end{cases}$$
For every $N\in \N$, let $x^N:G\to \R$ be a minimiser on $\mathcal{B}_N$, such that $x^N|_{(\mathcal{B}_N)^c}\equiv \tilde x|_{(\mathcal{B}_N)^c}$.
Then, as $N\to \infty$, $x^N$ converge along a subsequence to a global minimiser $\bar x$ of (\ref{rr}).
\end{proposition}

\begin{proof}
By lemma \ref{existence} such a minimiser $x^N$ exists and by lemma \ref{comparison}, $c_0\leq x^N_g\leq c_1$ for all $N\in \N$ and all $g\in G$. By Tychonov's theorem $[c_0,c_1]^{G}$ is a compact space w.r.t. pointwise convergence, so $x^N$ has a convergent subsequence $x^{N_k}\to \bar x$, which is then a global minimiser.

%\leq \rho 2^{-1}|S_1|\sum_{i\leq n} C_6 e^{\varepsilon i(D-\beta)}, $$ 
%which implies by the definition of a global minimiser that $$W_{\mathcal{B}_n}(x^{\rho,n})\leq \rho 2^{-1}|S_1|\sum_{i=1}^n C_6 e^{\varepsilon i(D-\beta)} \leq \rho 2^{-1}|S_1| C_6 e^{\varepsilon n(D-\beta)}(\sum_{i\in\N} e^{-i})\leq \rho |S_1|C_6 e^{\varepsilon n(D-\beta)}.$$
\end{proof}

\subsection{Minimal solutions and boundaries of sets}

In this section we investigate some less standard properties of minimal solutions on a set $\mathcal{B}$ with boundary values from the set $[c_0,c_1]$. Before proceeding with the analysis, we state the following simple proposition that quantifies the behaviour of $V$ around the absolute minima $c_0$ and $c_1$.

\begin{proposition}\label{V}
There exist constants $\rho_0, m_1>0$ and $b\geq1$, such that $4b\rho_0\leq c_1-c_0$ and such that it holds  for all $\rho \in (0,\rho_0]$:
\begin{itemize}
	\item[(a)] %On $ [c_0,c_0+2b \rho]$ is $V'$ a monotone increasing positive function, bounded from below by $V'(y)\geq m_0(y-c_0)$ and \\ 
	On $[c_1-2b \rho,c_1]$ is $V'$ a monotone increasing negative function, bounded from above by $m_1(y-c_1)\geq V'(y)$.
	\item[(b)] %$V(c_1-b y)\geq 
	$V(c_0+y)\geq V(c_1-b^{-1}y)$ for all $y\in [0,2b\rho]$.
	\item[(c)] For every $\rho\in (0,\rho_0]$ there exists a $\beta>0$, such that for all $y\in [c_0+2b \rho,c_1-2b \rho]$ $$V(y)\geq \max\{V(c_0+b \rho), V(c_1-b \rho)\}+\beta.$$
	\item[(d)] There exists a $\tilde \rho>0$ with $\tilde \rho<c_1-c_0-2b\rho$ such that $V(c_0+\tilde \rho)=V(c_1-2b \rho)$ and such that for all $y\in[c_0+\tilde \rho,c_1-2b \rho]$, $V(y)\geq V(c_0+\tilde \rho)$. Moreover, $2b\rho\leq (c_1-c_0-2b\rho-\tilde \rho)$.
\end{itemize}
\end{proposition}

\begin{proof}
Parts (a) and (b) follow from the non-degenericity condition on $c_0$ and $c_1$ and Taylor's theorem. Part (c) follows from the fact that $c_0,c_1\in \R$ are consecutive absolute minima and part (d) from both, Taylor's theorem and the fact that there are no other absolute minima in $[c_0,c_1]$.
\end{proof}

For the remainder of this section, let $x$ be a minimiser on a finite set $\mathcal{B}$, satisfying the boundary conditions given by a function $f:\Gamma\backslash \mathcal{B} \to [c_0,c_1]$. For $\underline{c}<\overline{c} \in [c_0,c_1]$ define the interval sets $$\mathcal{B}^{(\underline{c},\overline{c})}:=\{g\in \mathcal{B}\subset G \ | \ x_g\in (\underline{c},\overline{c})\}$$ and analogously for half-open and closed intervals from $\underline{c}$ to $\overline{c}$. Furthermore, define for $x:G\to \R$, $c\in \R$ and a set $\mathcal{D}\subset G$ the function $$P_{\mathcal{D}}(x,c):=\sum_{g\in \mathcal{D}}\Big(V(x_{ g})-V(c)\Big).$$ The following lemma gives an essential estimate, which combined with (\ref{IP}) and the growth bound (\ref{gob}) gives us the main lemma bellow (i.e. lemma \ref{main_lemma}).

\begin{lemma}\label{ts_estimate}Let $\mathcal{D}\subset G$ be a subset, the constants $\rho, b$ and $ m_1$ as in proposition \ref{V}, and let $x$ be a minimiser on a finite set $\mathcal{B}$, satisfying the boundary conditions given by a function $f:G\backslash\mathcal{B} \to [c_0,c_1]$. Define $\mathcal{B}^l:=\mathcal{B}^{[c_0,c_1-2b\rho)}$ and $\mathcal{B}^h:=\mathcal{B}^{[c_1-2b\rho, c_1-\rho)}$.%, and let $\tilde B\subset G$ denote the set $$\tilde B:=\{g\in \mathcal{B}^h\cap \p^{\out}\mathcal{D} \ | \ \forall \tilde g\in \mathcal{B}_1(g)\cap\mathcal{D}, \tilde g\in \mathcal{B}^h\}$$ 
Then there exist a constant $k_0>0$, such that $$
\#(\p^{\out}\mathcal{B}^l\cap \mathcal{D}^{(2\inn)})+P_{\mathcal{B}^h\cap \mathcal{D}^{\inn}}(x,c_1-\rho)
%\sum_{g\in \mathcal{B}^h\cap \mathcal{D}^{\inn}}(V(x_{g})-V(c_1-\rho))
\leq k_0 \Big(\#(\mathcal{B}^l\cap \p^{\f}\mathcal{D})+P_{\mathcal{B}^h\cap \p^{\out}\mathcal{D}}(x,c_1-\rho)\Big)
%\sum_{g\in \mathcal{B}^h\cap \p^{\out}\mathcal{D}}(V(x_{ g})-V(c_1-\rho))
.$$
\end{lemma}

\begin{proof}
Consider the following variation of $x$: $$\tilde x^\mathcal{D}_g:=\begin{cases}c_1- \rho & \text{ for } g\in  \mathcal{B}^{[c_0+\rho b,c_1-\rho)}\cap \mathcal{D},\\ c_1-b^{-1}(x_g-c_0) & \text{ for } g \in \mathcal{B}^{[c_0,c_0+\rho b)}\cap \mathcal{D},\\ x_g & \text{ else, } \end{cases}$$
 for which $\text{supp}(\tilde x^\mathcal{D}-x)\subset \mathcal{B}^{[c_0,c_1-\rho)}\cap \mathcal{D}$ holds. 
 
Define $\tilde{\mathcal{D}}:=(\mathcal{B}^{[c_0,c_1-\rho)}\cap \mathcal{D})^{\out}$. We split the proof of this lemma in three parts, each worked out below. In part 1, we will show that \begin{equation}\label{Eq1}W_{\tilde{\mathcal{D}}\cap \mathcal{D}^{\inn}}(x)- W_{\tilde{\mathcal{D}}\cap \mathcal{D}^{\inn}}(\tilde x^\mathcal{D})\geq 0.\end{equation} Since $x$ is a minimiser, $W_{\tilde{\mathcal{D}}}(\tilde x^\mathcal{D})- W_{\tilde{\mathcal{D}}}(x)\geq0,$ so by writing $$W_{\tilde{\mathcal{D}}\backslash \mathcal{D}^{\inn}}(\tilde x^\mathcal{D})- W_{\tilde{\mathcal{D}}\backslash \mathcal{D}^{\inn}}(x)= W_{\tilde{\mathcal{D}}}(\tilde x^\mathcal{D})- W_{\tilde{\mathcal{D}}}(x) + W_{\tilde{\mathcal{D}}\cap \mathcal{D}^{\inn}}(x)- W_{\tilde{\mathcal{D}}\cap \mathcal{D}^{\inn}}(\tilde x^\mathcal{D}),$$ we notice  that $$W_{\tilde{\mathcal{D}}\backslash \mathcal{D}^{\inn}}(\tilde x^\mathcal{D})- W_{\tilde{\mathcal{D}}\backslash \mathcal{D}^{\inn}}(x)\geq W_{\tilde{\mathcal{D}}\cap \mathcal{D}^{\inn}}(x)- W_{\tilde{\mathcal{D}}\cap \mathcal{D}^{\inn}}(\tilde x^\mathcal{D})$$
follows from (\ref{Eq1}). Let us write $\tilde{\mathcal{D}}$ as the disjoint union $\tilde{\mathcal{D}}=(\mathcal{B}^l\cap\tilde{\mathcal{D}})\cup (\mathcal{B}^h\cap\tilde{\mathcal{D}})$ and trivially estimate $$W_{\tilde{\mathcal{D}}\cap \mathcal{D}^{\inn}}(x)- W_{\tilde{\mathcal{D}}\cap \mathcal{D}^{\inn}}(\tilde x^\mathcal{D})\geq W_{\p^{\inn}\mathcal{B}^l\cap \mathcal{D}^{\inn}}(x)- W_{\p^{\inn}\mathcal{B}^l\cap \mathcal{D}^{\inn}}(\tilde x^\mathcal{D})+\sum_{g\in \mathcal{B}^h \cap \mathcal{D}^{\inn}}(V(x_g)-V(\tilde x^\mathcal{D})).$$ In the second part of the proof below we show that there exists a positive constant $\tilde{k}_1$, such that \begin{equation}\label{Eq2}W_{\p^{\inn}\mathcal{B}^l\cap \mathcal{D}^{\inn}}(x)-W_{\p^{\inn}\mathcal{B}^l\cap \mathcal{D}^{\inn}}(\tilde x^\mathcal{D})\geq \tilde{k}_1\#(\p^{\out}\mathcal{B}^{l}\cap (\mathcal{D}^{\inn})^{\inn}) .\end{equation} To obtain this estimate, no terms $V(x_g)$ with $g\in \mathcal{B}^h \cap \mathcal{D}^{\inn}$ were involved, so it follows since $\tilde{k}_1\leq 1$ that $$W_{\tilde{\mathcal{D}}\cap \mathcal{D}^{\inn}}(x)- W_{\tilde{\mathcal{D}}\cap \mathcal{D}^{\inn}}(\tilde x^\mathcal{D})\geq \tilde{k}_1\Big(\#(\p^{\out}\mathcal{B}^{l}\cap (\mathcal{D}^{\inn})^{\inn})+\sum_{g\in \mathcal{B}^h \cap \mathcal{D}^{\inn}}(V(x_g)-V(c_1-\rho))\Big).$$ 
In the third part below, we prove that there exists a positive constant $\tilde{k}_2$, such that $$\tilde{k}_2\Big(\#(\mathcal{B}^l\cap \p^{\f}\mathcal{D})+\sum_{g\in \mathcal{B}^h\cap \p^{\out}\mathcal{D}}(V(x_{g})-V(c_1-\rho))\Big) \geq W_{\tilde{\mathcal{D}}\backslash \mathcal{D}^{\inn}}(\tilde x^\mathcal{D})- W_{\tilde{\mathcal{D}}\backslash \mathcal{D}^{\inn}}(x).$$

\noindent
{\it Part 1:}\\
Part (b) of proposition \ref{V} implies that $V(c_1-b^{-1}(x_g-c_0))\leq V(x_g)$ for all $g\in \mathcal{B}^{[c_0,c_0+\rho b]}$. Moreover, since $V(c_1-\rho)\leq V(c_0+b\rho)$, it follows from parts (a) and (c) of the same proposition that $V(x_g)\geq V(c_1-\rho)$ for all $g \in \mathcal{B}^{[c_0+\rho b,c_1-\rho)}$. Together this gives
\begin{equation}\label{p1}V(x_g)\geq V(\tilde x^\mathcal{D}_g)\text{ for all } g\in G. \end{equation}

Next, we prove that for all $g\in \mathcal{D}^{\inn}$, 
\begin{equation}\label{g1}|\nabla_g(x)|^2\geq |\nabla_g(\tilde x^\mathcal{D})|^2. \end{equation}
Indeed, it holds for all $g\in \mathcal{D}^{\inn}$ and $\tilde g\in \mathcal{B}_1(g)$ that $\tilde g\in \mathcal{D}$, and it is enough to show that for all pairs $(g,\tilde g)\in \mathcal{D}\times \mathcal{D}$ with $d(g,\tilde g)=1$, $(x_{\tilde g}-x_g)^2\geq (\tilde x^\mathcal{D}_{\tilde g}-\tilde x^\mathcal{D}_g)^2$. To show this, we do a case study. 
\begin{itemize}
\item First, observe that this inequality is obviously satisfied in the case where $g,\tilde g \in (\mathcal{B}^{[c_0+\rho b,c_1-\rho)}\cup \mathcal{B}^{[c_1-\rho,c_1]})\cap \mathcal{D}$, and since $b\geq 1$ also when $g,\tilde g \in \mathcal{B}^{[c_0,c_0+\rho b)}\cap \mathcal{D}$. 
\item In case that $g\in \mathcal{B}^{[c_0,c_0+\rho b)}\cap \mathcal{D}$ and $\tilde g \in \mathcal{B}^{[c_0+\rho b,c_1-\rho)}\cap \mathcal{D}$, the definition of $\tilde x^\mathcal{D}$ gives us $b^{-1}(x_g-c_0)=c_1-\tilde{x}^\mathcal{D}_g$ and $c_1-b^{-1}(x_{\tilde g}-c_0)\geq c_1-\rho=\tilde{x}^\mathcal{D}_{\tilde g}$, so $$(x_{\tilde g}-x_g)\geq b^{-1}(x_{\tilde g}-x_g)\geq (c_1-b^{-1}(x_{\tilde g}-c_0)-c_1+b^{-1}(x_{g}-c_0)) \geq(\tilde{x}^\mathcal{D}_{\tilde g}-\tilde{x}^\mathcal{D}_{g})$$ and since all the terms in these inequalities are positive, the statement follows. 
\item Next, we consider the case where $g\in \mathcal{B}^{[c_0,c_0+\rho b)}\cap \mathcal{D}$ and $\tilde g \in \mathcal{B}^{[c_1-\rho ,c_1]}\cap \mathcal{D}$, which holds since $x_{\tilde g}-x_g\geq (c_1-c_0-\rho(b+1))$ while $(\tilde{x}^\mathcal{D}_{\tilde g}-\tilde{x}^\mathcal{D}_{g})\leq \rho$ and $\rho_0$ may be assumed smaller if necessary. 
\item Finally, we consider the case where $g\in \mathcal{B}^{[c_0+\rho b, c_1-\rho)}\cap \mathcal{D}$ and $\tilde g \in \mathcal{B}^{[c_1-\rho ,c_1]}\cap \mathcal{D}$, where $x_{\tilde g}=\tilde{x}^\mathcal{D}_{\tilde g}$ and $x_g\leq c_1-\rho$, which gives $(x_g-x_{\tilde g})\geq (\tilde{x}^\mathcal{D}_{g}-\tilde{x}^\mathcal{D}_{\tilde g})\geq 0$.
\end{itemize}
This proves (\ref{g1}), which combined with (\ref{p1}) implies (\ref{Eq1}).\\

\noindent
{\it Part 2:}\\
To prove (\ref{Eq2}), let $g\in \p^{\out}\mathcal{B}^{[c_0,c_1-2b\rho)}\cap (\mathcal{D}^{\inn})^{\inn}$. Then one of the following two cases are satisfied:
\begin{itemize}
\item There exists a $\tilde g\in \mathcal{B}_1(g)$ with $\tilde g\in \mathcal{B}^{[c_0+2b\rho,c_1-2b\rho)}\cap \mathcal{D}^{\inn}$, which by proposition \ref{V} (c) implies that $V(x_{\tilde g})-V(\tilde x^{\mathcal{D}}_{\tilde g})\geq \beta>0$.
\item There exists a $\tilde g\in \mathcal{B}_1(g)$ with $\tilde g\in \mathcal{B}^{[c_0,c_0+2b\rho)}\cap \mathcal{D}^{\inn}$ while $ g\in \mathcal{B}^{[c_1-2b\rho,c_1]}\cap \mathcal{D}$. It then follows that $(x_{ g}-x_{\tilde g})\geq (c_1-c_0-4b\rho)$ and $|\tilde x^{\mathcal{D}}_{ g}-\tilde x^{\mathcal{D}}_{\tilde g}|\leq \rho$, since for all $g\in D$, $\tilde x^{\mathcal{D}}_{g}\in [c_1-\rho, c_1]$. 
\end{itemize}
This implies together with (\ref{p1}), (\ref{g1}) and the considerations above, that for every $g\in \p^{\out}\mathcal{B}^{[c_0,c_1-2b\rho)}\cap (\mathcal{D}^{\inn})^{\inn}$ there exists a $\tilde g\in \mathcal{D}^{\inn}$, such that $(1/4(|\nabla_{\tilde g}(x)|^2-|\nabla_{\tilde g}(\tilde x^\mathcal{D})|^2)+V(x_{\tilde g})-V(\tilde x^\mathcal{D}_{\tilde g})\geq S^{-1} \cdot \tilde{k}_1$ for $\tilde{k}_1:=\min\{1,\beta, ((c_1-c_0-4b\rho)^2-\rho^2)/4\}>0$. Taking into account possible multiplicities of the points $g$ for a specific point $\tilde g$ gives us the estimate.\\

\noindent
{\it Part 3:}\\
Let $g\in \tilde{\mathcal{D}}\backslash \mathcal{D}^{\inn}$. By (\ref{p1}) $$\sum_{g\in \tilde{\mathcal{D}}\backslash \mathcal{D}^{\inn}}\frac{1}{4}\left(|\nabla_g(\tilde x^\mathcal{D})|^2- |\nabla_g(x)|^2\right)\geq W_{ \tilde{\mathcal{D}}\backslash \mathcal{D}^{\inn}}(\tilde x^\mathcal{D})- W_{ \tilde{\mathcal{D}}\backslash \mathcal{D}^{\inn}}(x)$$ and by part 1, for all $(g,\tilde g)\in \mathcal{D}\times \mathcal{D}$ with $d(g,\tilde g)=1$, $(\tilde x^\mathcal{D}_g-\tilde x^\mathcal{D}_{\tilde g})^2\leq (x_g-x_{\tilde g})^2$ holds. So we only need to consider $(g,\tilde g)\in \p^{\inn}\mathcal{D}\times \p^{\out}\mathcal{D} \cup \p^{\out}\mathcal{D}\times \p^{\inn}\mathcal{D}$, which after reorganising the pairs gives $$\sum_{g\in \p^{\inn}\mathcal{D}}\sum_{\tilde g \in \p^{\out}\mathcal{D}\cap \mathcal{B}_1(g)}\left((\tilde x^\mathcal{D}_g-\tilde x^\mathcal{D}_{\tilde g})^2- (x_g-x_{\tilde g})^2\right)\geq\sum_{g\in \tilde{\mathcal{D}}\backslash \mathcal{D}^{\inn}}\frac{1}{4}\left(|\nabla_g(\tilde x^\mathcal{D})|^2- |\nabla_g(x)|^2\right),$$ where  $\tilde x^\mathcal{D}_{\tilde g}=x_{\tilde g}$. % 4 because 2 times the terms also come up in the gradient
We consider first the case where for $(g,\tilde g)\in \p^{\inn} \mathcal{D}\times \p^{\out}\mathcal{D}$ either $g\in \mathcal{B}^l$ or $\tilde g\in \mathcal{B}^l$. If $g\in \mathcal{B}^l$, we use the simple estimate  \begin{equation}\label{p2}\sum_{\tilde g \in \mathcal{B}_1(g)}\left((\tilde x^{\mathcal{D}}_{\tilde g}-\tilde x^{\mathcal{D}}_g)^2-(x_{\tilde g}-x_g)^2\right)\leq \sum_{\tilde g \in \mathcal{B}_1(g)}( x_{\tilde g}-\tilde x^{\mathcal{D}}_g)^2 \leq S(c_1-c_0)^2\end{equation}  and similarly, if $\tilde g \in \mathcal{B}^l\cap \p^{\out}\mathcal{D}$, we use the same estimate with the roles of $\tilde g$ and $g$ reversed.  
We are left with the case where $g\in \mathcal{B}^h\cap \p^{\inn}\mathcal{D}$, and it holds that either $\tilde g\in \mathcal{B}^{[c_1-\rho,c_1]}$, which we can forget, since then $$(\tilde x^{\mathcal{D}}_{\tilde g}-\tilde x^{\mathcal{D}}_{g})^2=(x_{\tilde g}-(c_1-\rho))^2\leq (x_{\tilde g}-x_g)^2, $$ or $\tilde g\in \mathcal{B}^h\cap \p^{\out}\mathcal{D}$. For this second case we observe that $$(\tilde x^{\mathcal{D}}_{g}-\tilde x^{\mathcal{D}}_{\tilde g})^2-(x_{\tilde g}-x_g)^2\leq (\tilde x^{\mathcal{D}}_{g}-\tilde x^{\mathcal{D}}_{\tilde g})^2=(c_1-\rho-x_{\tilde g})^2\leq ((c_1-x_{\tilde g})^2-\rho^2),$$ which we may by condition (a) from proposition \ref{V} estimate by $$(c_1-x_{\tilde g})^2-\rho^2=2 \int_{x_{\tilde g}}^{c_1-\rho}(c_1-y) dy\leq 2m_1^{-1}\int_{x_{\tilde g}}^{c_1-\rho}(-V'(y))dy=2m_1^{-1}(V(x_{\tilde g})-V(c_1-\rho)).$$
Together with (\ref{p2}) this gives with $\tilde{k}_2:=\max\{S(c_1-c_0)^2, 2Sm_1^{-1}\}$ the inequality $$\tilde{k}_2\Big(\#(\mathcal{B}^l\cap \p^{\f}\mathcal{D})+\sum_{g\in \mathcal{B}^h\cap \p^{\out}\mathcal{D}}(V(x_{g})-V(c_1-\rho))\Big) \geq W_{\tilde{\mathcal{D}}\backslash \mathcal{D}^{\inn}}(\tilde x^\mathcal{D})- W_{\tilde{\mathcal{D}}\backslash \mathcal{D}^{\inn}}(x).$$
\end{proof}

Note that for $\rho\to0$, it follows since the absolute minima $c_0$ and $c_1$ are non-degenerate that then also $\beta\to0$, where $\beta$ is as in proposition \ref{V}. This in turn implies that the constant $\tilde{k}_1$ from the proof above converges to zero with $\rho$, and so $k_0\to \infty$ for $\rho\to 0$.

We call a set of vertices $\mathcal{D}\subset G$ connected, if there exists a corresponding connected subset $\tilde{\mathcal{D}}\subset \Gamma$ such that $\mathcal{D}\subset \tilde{\mathcal{D}}$ and such that $\tilde{\mathcal{D}}\cap G=\mathcal{D}$. (We view edges as open intervals and adding the adjoining vertices corresponds to closing the intervals.)

\begin{lemma}\label{connected_components}Let $x$ be a minimiser on a finite set $\mathcal{B}$ satisfying the boundary conditions given by a function $f:G\backslash\mathcal{B} \to [c_0,c_1]$ and let the sets $\mathcal{B}^l,\mathcal{B}^h$ be as in lemma \ref{ts_estimate}. Then for every connected component $\mathcal{D}$ of $\mathcal{B}^l$, $$\mathcal{D} \cap \p^{\out} \mathcal{B}\neq \varnothing.$$
\end{lemma}

\begin{proof}
Assume that there is a connected component $\mathcal{D}\subset \mathcal{B}^{l}$ which does not intersect the boundary $\p^{\out} \mathcal{B}$. Then $x_g\in [c_1-2b\rho,c_1]$ on $\p^{\out}\mathcal{D}$.

Recall the definition of $\tilde \rho$ from proposition \ref{V} and %define %constants $a:=2b\rho$ and $b:=(c_1-c_0+\tilde \rho+2b\rho)$. We 
define a variation $\tilde x$ of $x$ on $\mathcal{D}^0$ by $$\tilde x_g:=\min\Big\{c_1, \frac{2b\rho}{(c_1-c_0+\tilde \rho+2b\rho)}((c_1-2b\rho)-x_g)+(c_1-2b\rho)\Big\}.$$ Intuitively, this is a reflexion of $x$ over $c_1-2b\rho$ with a contracting factor $\frac{2b\rho}{(c_1-c_0+\tilde \rho+2b\rho)}<1$ (by part (d) of proposition \ref{V}).

It holds that $\tilde x_g\in [c_1-2b\rho,c_1]$ for all $g\in \mathcal{D}$. More precisely, it holds that $\tilde x_g\in [c_1-2b\rho,c_1)$ for all $g\in \mathcal{D}^{(c_0+\tilde \rho, c_1-2b\rho]}$ and $\tilde x_g= c_1$ for all $g\in \mathcal{D}^{[c_0,c_0+\tilde \rho]}$. It follows by part (d) of proposition \ref{V} that $V(x_g)\geq V(\tilde x_g)$ for all $g$ and, since $\frac{2b\rho}{(c_1-c_0+\tilde \rho+2b\rho)}<1$, it also holds that $|\nabla_g(x)|^2>|\nabla_g(\tilde x)|^2$ whenever $\nabla_g(x)\neq \nabla_g(\tilde x)$. This implies that $W_{\mathcal{D}^{\out}}(x)> W_{\mathcal{D}^{\out}}(\tilde x)$ whenever $\tilde x\neq x$. By the definition of $\mathcal{B}^l$ this is indeed the case, which gives us a contradiction.
\end{proof}

\section{Controlling the asymptotics}\label{dp}

\subsection{Main lemma}

Let $b,\rho_0$ be as in proposition \ref{V} and fix a constant $\rho \in (0,\rho_0]$, which gives the notion of sets $\mathcal{B}_N^h, \mathcal{B}_N^l$, defined as in lemma \ref{ts_estimate} for balls $\mathcal{B}_N$. In this section we prove the main technical result about how these sets behave when $N$ goes to infinity. 

\begin{definition}\label{rini}

Let us recall the definition of the growth rate $D$ from (\ref{gob}) and define for $r>0$ and $\xi_0\in \p \Gamma$ the following objects:

\begin{itemize}

\item 
The sequences $r_i\to r$ and $d_i\to 0$ for $i\geq 1$ by \begin{equation}\label{ri}r_i:=\frac{6 r}{\pi^2}\sum_{j=1}^i\frac{1}{j^2}\ \text{ and } \ d_i:=r_{i+1}-r_i=\frac{6 r}{\pi^2 (i+1)^2} \ .\end{equation}  

\item 
For a given natural number $n_1$ the increasing sequence of real numbers \begin{equation}\label{ni}n_{i+1}:=\left(\frac{D+1/2}{D+1/4}\right)n_i=\left(\frac{D+1/2}{D+1/4}\right)^{i-1}n_1\ .\end{equation} 

\item 
For $r_i$ and $n_i$ as above \begin{equation}\label{vi}\mathcal{V}_i:=(\mathcal{C}_{B^\varepsilon_{r_{i+1}}(\xi_0)}\backslash \mathcal{C}_{B^\varepsilon_{r_i}(\xi_0)})\backslash \mathcal{B}_{\lfloor n_i\rfloor} \ .\end{equation} where $\lfloor \cdot \rfloor$ denotes the floor function. 
\end{itemize}

\end{definition}

\begin{remark}\label{vi_structure}
Let $r_i$, $n_i$ and $\mathcal{V}_i$ be as in the definition above. Recall the definition of the constant $t_n$ from lemma \ref{separating_sets} and rewrite it for every $n\in \N$ as  \begin{equation}\label{konstants}t_n= 4k_1^{-1} e^{-\varepsilon n} \ \text{ where } \ k_1:=(4\min\{C_2,4\varepsilon^{-1}e^{4\varepsilon}\})^{-1}.\end{equation} Moreover, recall from lemma \ref{separating_sets} the sets $\mathcal{A}_{r,t_n}$ given by $A_{r+2t_n}^{t_n}(\xi_0)\subset\p \Gamma$, which ``separate'' $\mathcal{C}_{B^\varepsilon_r(\xi_0)}$ and $\mathcal{C}_{\Gamma\backslash B^\varepsilon_{r+4t_n}(\xi_0)}$ outside $\mathcal{B}_n$. Using definitions (\ref{vi}) and (\ref{konstants}) we may, for every $i\in \N$, write the set $\mathcal{V}_i$ as a disjoint union $\mathcal{V}_i=\bigsqcup_{j\in I_i}\tilde{\mathcal{A}}_{i_j}$, where the sets $\tilde{\mathcal{A}}_{i_j}$ satisfy $\mathcal{A}_{r_{i_j},t_{\lfloor n_i\rfloor}} \subset \tilde{\mathcal{A}}_{i_j}$. It follows from definitions of $r_i$ and $t_n$ that \begin{equation}\label{eq1}\# I_i\geq  \left\lfloor \frac{k_1d_i}{ e^{-\varepsilon \lfloor n_i\rfloor}}\right\rfloor-1 \geq k_1 d_i e^{-\varepsilon}e^{\varepsilon n_i}-2.\end{equation} 
\end{remark}

The size of $I_i$ measures the number of disjoint separating sets outsize a ball of radius $\lfloor n_i\rfloor$ that are contained in $\mathcal{V}_i$. Note the exponential growth of $\# I_i$ with $n_i$. %This is an essential part to the proof of the following lemma, which gives us a powerful estimate about the growth of the set $\mathcal{C}_{B^\varepsilon_{r_i}(\xi_0)}\cap \mathcal{B}_N^{c_1}$. The idea behind the proof is the following. From lemma \ref{ts_estimate} we deduce that if $\mathcal{C}_{B^\varepsilon_{r_i}(\xi_0)}\cap \p^{\out} \mathcal{B}_N^{c_1}$ is large, then for all $j\in I_i$, the set $\p^{\inn} \mathcal{C}_{B^\varepsilon_{r_{i_j}}(\xi_0)}\cap \mathcal{B}_N^{c_1}$ can be at most by a factor smaller. By fitting $\p^{\inn} \mathcal{C}_{B^\varepsilon_{r_{i_j}}(\xi_0)}$ into separating sets $\tilde{\mathcal{A}}_{i_j}$ and by using the isoperimetric inequality, the size of $I_i$ implies that $\mathcal{V}_i\cap \mathcal{B}_N^{c_1}$ is exponentially larger than $\mathcal{C}_{B^\varepsilon_{r_i}(\xi_0)}\cap \mathcal{B}_N^{c_1}$. In particular, also $\mathcal{C}_{B^\varepsilon_{r_{i+1}}(\xi_0)}\cap \mathcal{B}_N^{c_1}$ is exponentially larger and the strategy can be repeated by induction.

\begin{lemma}[Main lemma]\label{main_lemma}
Let $r>0$ and $\xi_0\in \p \Gamma$ be given and let $C_D$, $C_0$, $k_0$ and $k_1$  come from (\ref{gob}), the isoperimetric inequality (\ref{IP}), lemma \ref{ts_estimate} and (\ref{konstants}), respectively. Define $k_2$ by \begin{equation}\label{k_2}k_2:=\max\{C_D,\left(6 k_0 C_D C_0 \right)^{(\frac{4D+1}{4D})}\}. \end{equation} 
Let furthermore $r_i$, $n_i$ and $\mathcal{V}_i$ be as in definition \ref{rini}, where $n_1$ is a real number satisfying
\begin{equation}\label{n1}n_1\geq k_3:=\max\left\{\ 4 (4D+1)\varepsilon^{-1}\ , \ \frac{2}{\varepsilon}\log\left(\frac{(k_2+2C_D)4\pi^2}{6r C_Dk_1e^{-\varepsilon}}\right)\right\}.\end{equation}

Then, whenever there exist for a minimiser $x^N$ on $\mathcal{B}_N$ an $n_i\leq N$ satisfying \begin{equation}\label{ih2}\#(\mathcal{C}_{B^\varepsilon_{r_i}(\xi_0)}\cap \mathcal{B}_N^{l})+ P_{\mathcal{C}_{B^\varepsilon_{r_i}(\xi_0)}\cap \mathcal{B}_N^{h}}(x^N,c_1-\rho)\geq k_2e^{\varepsilon (D+\frac{1}{4}) n_i},\end{equation} 
 it follows for all integers $\iota\geq i$ with $n_\iota<N$ that $$\#(\mathcal{C}_{B^\varepsilon_{r_\iota}(\xi_0)}\cap \mathcal{B}_N^{l})+ P_{\mathcal{C}_{B^\varepsilon_{r_\iota}(\xi_0)}\cap \mathcal{B}_N^{h}}(x^N,c_1-\rho)\geq k_2e^{\varepsilon (D+\frac{1}{4}) n_\iota}.$$
\end{lemma}

\begin{proof}
By inclusion, (\ref{ih2}) holds for all $j\in I_i$ where $I_i$ is as in remark \ref{vi_structure} the index set that gives $r_{i_j}$ and $\tilde{\mathcal{A}}_{i_j}$ and for which $\mathcal{V}_i=\bigsqcup_{j\in I_i}\tilde{\mathcal{A}}_{i_j}$. In case that one of the two terms on the left in (\ref{ih2}) is smaller than one, we may compensate that term in the constants on the right and work only with the other term from here on. Otherwise, observe that for all $0\leq \alpha\leq 1$ and $a,b\geq 1$, $$a^{\alpha}+b\geq a^\alpha+b^{\alpha}\geq (a+b)^\alpha.$$ By raising (\ref{ih2}) to the power $\frac{4D}{4D+1}$ and by (\ref{k_2}) we thus conclude that $$\Big(\#(\mathcal{C}_{B^\varepsilon_{r_i}(\xi_0)}\cap \mathcal{B}_N^{l})\Big)^{\frac{4D}{4D+1}}+ \Big(P_{\mathcal{C}_{B^\varepsilon_{r_i}(\xi_0)}\cap \mathcal{B}_N^{h}}(x^N,c_1-\rho)\Big)\geq 6k_0C_DC_0e^{\varepsilon D n_i}.$$ 
%By definition of $L_0$ it follows for every $L\geq L_0$ that $$ \frac{L}{ \log(\alpha L)}> L^{\left(\frac{4D}{4D+1}\right)}$$ 
By the isoperimetric inequality (\ref{IP}) and %for any set $\mathcal{D}$ with $\# \mathcal{D}\geq L_0$, \begin{equation}\label{ve} C_0\#( \p^{\out} \mathcal{D})\geq (\# \mathcal{D})^{\left(\frac{4D}{4D+1}\right)}.\end{equation}
since $C_0\geq 1$ it follows: $$ \#( \p^{\out} ( \mathcal{C}_{B^\varepsilon_{r_{i_j}}(\xi_0)}\cap \mathcal{B}_N^{l})+P_{\mathcal{C}_{B^\varepsilon_{r_i}(\xi_0)}\cap \mathcal{B}_N^{h}}(x^N,c_1-\rho)\geq 6k_0C_De^{\varepsilon D n_i}. $$
Since $ (\mathcal{B}_N^{l})^{\out}=\mathcal{B}_N^{l}\cup \p^{\out}\mathcal{B}_N^{l}$, the following estimates hold:
\begin{equation*}\begin{aligned} &\#( \p^{\out} ( \mathcal{C}_{B^\varepsilon_{r_{i_j}}(\xi_0)}\cap \mathcal{B}_N^{l}))\leq \\ \leq& \#( \p^{\out} \mathcal{C}_{B^\varepsilon_{r_{i_j}}(\xi_0)}\cap (\mathcal{B}_N^{l})^{\out})+ \#( (\mathcal{C}_{B^\varepsilon_{r_{i_j}}(\xi_0)})^{\out}\cap  \p^{\out}\mathcal{B}_N^{l})\\
\leq & \#( \p^{\out} \mathcal{C}_{B^\varepsilon_{r_{i_j}}(\xi_0)}\cap \mathcal{B}_N^{l})+ 2 \#( (\mathcal{C}_{B^\varepsilon_{r_{i_j}}(\xi_0)})^{\out}\cap  \p^{\out}\mathcal{B}_N^{l}).\end{aligned}\end{equation*} 
Thus $$\#( \p^{\out} \mathcal{C}_{B^\varepsilon_{r_{i_j}}(\xi_0)}\cap \mathcal{B}_N^{l})+(*)\geq 6k_0C_De^{\varepsilon D n_i},$$ where $$(*):=2 \#( (\mathcal{C}_{B^\varepsilon_{r_{i_j}}(\xi_0)})^{\out}\cap  \p^{\out}\mathcal{B}_N^{l})+P_{\mathcal{C}_{B^\varepsilon_{r_i}(\xi_0)}\cap \mathcal{B}_N^{h}}(x^N,c_1-\rho).$$ By lemma \ref{separating_sets} $\p^{\out} \mathcal{C}_{B^\varepsilon_{r_{i_j}}(\xi_0)}\subset \tilde{\mathcal{A}}_{i_j} \cup \mathcal{B}_{\lfloor n_i\rfloor}$ and thus by (\ref{gob}) 
$$\#(\tilde{\mathcal{A}}_{i_j}\cap \mathcal{B}_N^{l})+(*)\geq C_D(6k_0-1)e^{\varepsilon D n_i}.$$ %since it is clear from the proof of lemma \ref{ts_estimate} that $k_0\geq1$.

We estimate $(*)$ from above by lemma \ref{ts_estimate} and obtain $$2k_0 \Big(\#(\mathcal{B}_N^l\cap \p^{\f}(\mathcal{C}_{B^\varepsilon_{r_{i_j}}(\xi_0)})^{(3\out)})+P_{\mathcal{B}^h\cap \p^{\out}(\mathcal{C}_{B^\varepsilon_{r_{i_j}}(\xi_0)})^{(3\out)}}(x,c_1-\rho)\Big)\geq (*).$$ By lemma \ref{separating_sets} $ \p^{\f}(\mathcal{C}_{B^\varepsilon_{r_{i_j}}(\xi_0)})^{(3\out)})\subset \tilde{\mathcal{A}}_{i_j} \cup \mathcal{B}_{\lfloor n_i\rfloor}$ and for small enough $\rho_0>0$, proposition \ref{V} gives $$P_{\mathcal{B}^h\cap \mathcal{B}_{n_i}}(x,c_1-\rho)\leq \#\mathcal{B}_{n_i}(V(c_1-2b\rho)-V(c_1-\rho))\leq \#\mathcal{B}_{n_i},$$ so by (\ref{gob}) \begin{equation}\label{eq5}(2k_0+1)\#(\tilde{\mathcal{A}}_{i_j}\cap \mathcal{B}_N^{l})+2k_0P_{\mathcal{B}^h\cap \tilde{\mathcal{A}}_{i_j}}(x,c_1-\rho) +4C_Dk_0\geq C_D(6k_0-1)e^{\varepsilon D n_i}.\end{equation} Since it is clear from the proof of lemma \ref{ts_estimate} that $k_0\geq1$, it holds $$\#(\tilde{\mathcal{A}}_{i_j}\cap \mathcal{B}_N^{l})+P_{\mathcal{B}^h\cap \tilde{\mathcal{A}}_{i_j}}(x,c_1-\rho)\geq C_De^{\varepsilon D n_i}.$$

It follows from $\mathcal{V}_i=\bigsqcup_{j\in I_i}\tilde{\mathcal{A}}_{i_j}$ and  (\ref{eq1}) that \begin{equation}\label{eq4}\Big(\#(\mathcal{V}_i\cap \mathcal{B}_N^{l})+P_{\mathcal{V}_i\cap \mathcal{B}_N^{h}}(x^N,c_1-\rho)\Big)\geq \frac{C_D 6r k_1e^{-\varepsilon}}{\pi^2(i+1)^2}e^{\varepsilon (D+1) n_i}-2C_D e^{\varepsilon D n_i}=:(\star).\end{equation} 

Rewriting the right-hand side by $$(\star)=C_D\left(\frac{6r k_1e^{-\varepsilon}} {\pi^2(i+1)^2}e^{\varepsilon \frac{n_i}{2}} -2  e^{-\varepsilon \frac{n_i}{2}} \right) e^{\varepsilon (D+\frac{1}{2}) n_{i}}$$ we use the definition of $n_{i+1}$ in (\ref{ni}) to write $(D+1/4)n_{i+1}=(D+1/2)n_i$ and obtain \begin{equation}\label{eq3}(\star)= C_D\left(\frac{6r  k_1e^{-\varepsilon}} {\pi^2(i+1)^2}e^{\varepsilon \frac{n_1}{2}\left(\frac{D+1/2}{D+1/4}\right)^{i-1}} -2 \right) e^{\varepsilon (D+\frac{1}{4}) n_{i+1}}.\end{equation}
Let $a:=(1+4D)^{-1}>0$, so that $\left(\frac{D+1/2}{D+1/4}\right)=1+a$ and estimate $$\frac{1} {(i+1)^2}e^{\varepsilon \frac{n_1}{2}(1+a)^{i-1}} \geq e^{\varepsilon \frac{n_1}{2}}\left(\frac{e^{\varepsilon \frac{n_1}{2}(i-1)a}}{(i+1)^2}\right).$$
By condition (\ref{n1}), we chose $n_1$ large enough that $$\frac{6r  C_D k_1 e^{-\varepsilon}}{\pi^2}e^{\varepsilon\frac{n_1}{2}}-2 C_D \geq k_2.$$ Observe that the function $(1+x)^{-2}e^{\beta (x-1)}$ is monotone increasing for all $x\geq 0$, whenever $\beta\geq 2$. In particular, since $n_1\geq 4 (4D+1)\varepsilon^{-1}$, it follows that $a \varepsilon n_1>4$, so for all $j\geq 1$ $$\frac{e^{\varepsilon \frac{n_1}{2}(j-1)a}}{(j+1)^2} \geq \frac{1}{4}.$$
Hence, it follows from inequality (\ref{eq4}) by (\ref{eq3}) and by condition (\ref{n1})
$$(\star)\geq C_D \left(6r\pi^{-2} k_1 e^{-\varepsilon}e^{\varepsilon \frac{n_1}{2}}\left(\frac{e^{\varepsilon \frac{n_1}{2}(i-1)\alpha}}{(i+1)^2}\right) -2 \right) e^{\varepsilon (D+\frac{1}{4}) n_{i+1}}\geq k_2 e^{\varepsilon (D+\frac{1}{4}) n_{i+1}}.$$ 
 
Since $(\mathcal{V}_i\cap \mathcal{B}_N^{l})\subset(\mathcal{C}_{B^\varepsilon_{r_{i+1}}(\xi_0)}\cap \mathcal{B}_N^{l})$ we have proved the statement of the lemma for $\iota=i+1$ and we obtain the full statement by induction.
\end{proof}

With this lemma we easily obtain the following corollary:

\begin{corollary}\label{cor}
Let $B_{r}^\varepsilon(\xi_0)\subset D_1$ be a ball at infinity, for $N\in \N$  let $x^N$ be minimisers solving the minimisation problems defined in proposition \ref{local_minimisers} and let $k_2$ and $k_3$ be as in the main lemma above. Then $$\#(\mathcal{C}_{B^\varepsilon_{r_1}(\xi_0)}\cap \mathcal{B}_N^{l}) \leq \bar n,$$ where $$\bar n:=\lceil k_2e^{\varepsilon (D+\frac{1}{2})k_3}\rceil$$ and $\lceil \cdot\rceil$ denotes the ceiling function

%By defining $$\bar n:=\lceil k_2e^{\varepsilon (D+\frac{1}{2})k_3}\rceil,$$ where $\lceil \cdot\rceil$ denotes the ceiling function, it follows $$\#(\mathcal{C}_{B^\varepsilon_{r_1}(\xi_0)}\cap \mathcal{B}_N^{l}) \leq \bar n.$$ 
\end{corollary}

\begin{proof}
%Define $n_1:=\Big(\frac{D+1/2}{D+1/4}\Big)k_3$. 
Assume that the corollary is not true and let $N\in \N$ be any fixed number larger than $k_3$. It is easy to see that there exists a real number $$n_1\in \Big[k_3,\Big(\frac{D+1/2}{D+1/4}\Big)k_3\Big],$$ such that $n_i=N$ for some $i\geq 1$. It then follows 
$$\#(\mathcal{C}_{B^\varepsilon_{r_1}(\xi_0)}\cap \mathcal{B}_N^{l}) \geq k_2e^{\varepsilon (D+\frac{1}{4}) n_1}$$ and since $ n_1\geq k_3,$ it holds by lemma \ref{main_lemma}, that for every $j\geq 1$ $$ \#(\mathcal{C}_{B^\varepsilon_{r_j}(\xi_0)}\cap \mathcal{B}_N^{l})+P_{\mathcal{C}_{B^\varepsilon_{r_j}(\xi_0)}\cap \mathcal{B}_N^{h}}(x,c_1-\rho)\geq  k_2e^{\varepsilon(D+\frac{1}{4})n_j}$$ and in particular $$\#(\mathcal{C}_{B^\varepsilon_{r_i}(\xi_0)}\cap \mathcal{B}_N^{l})+P_{\mathcal{C}_{B^\varepsilon_{r_i}(\xi_0)}\cap \mathcal{B}_N^{h}}(x,c_1-\rho)\geq  k_2e^{\varepsilon(D+\frac{1}{4})n_i}= k_2e^{\varepsilon(D+\frac{1}{4})N} .$$ On the other hand, as in (\ref{eq5}), (\ref{gob}) gives $$\#(\mathcal{C}_{B^\varepsilon_{r_i}(\xi_0)}\cap \mathcal{B}_N^{l})+P_{\mathcal{C}_{B^\varepsilon_{r_i}(\xi_0)}\cap \mathcal{B}_N^{h}}(x,c_1-\rho) \leq  \#\mathcal{B}_N\leq C_D e^{\varepsilon D N},$$ which is a contradiction, since $k_2\geq C_D$. 
\end{proof}

\subsection{Proof of the Existence theorem}\label{dp_section}
In this section we prove the main theorem. Its proof is basically contained in the following two lemmas. The first is a rather direct consequence of corollary \ref{cor} and lemma \ref{connected_components}. It implies that when $B_r^\varepsilon(\xi_0)\subset D_1$, the values of minimisers $x^N$ on $\mathcal{C}_{B^\varepsilon_{r_1}(\xi_0)}^{(\bar n \inn)}$ are uniformly close to $c_1$. 

\begin{lemma}\label{values_at_infinity} Let $B_r^{\varepsilon}(\xi_0)\subset D_1$ and $x^N$ for $N\in \N$ be minimisers solving the corresponding minimisation problems defined in proposition \ref{local_minimisers}. Then it holds for $\bar n$ defined in corollary \ref{cor} that $$\mathcal{C}_{B^\varepsilon_{r_1}(\xi_0)}^{(\bar n \inn)}\cap \mathcal{B}_N^{l}=\varnothing.$$
\end{lemma}

\begin{proof}
Assume that the lemma is not true, i.e. there exists a $g\in \mathcal{C}_{B^\varepsilon_{r_1}(\xi_0)}^{(\bar n \inn)}\cap \mathcal{B}_N^{l}$. By lemma \ref{connected_components}, there exists a point $\tilde g\in \p^{\out}(\mathcal{B}_N^{l}\cap \p^{\out}\mathcal{B}_{N})$ and a path $p_{g,\tilde g}\subset \Gamma$ from $g$ to $\tilde g$, with $(p_{g,\tilde g}\cap G)\subset \mathcal{B}_N^{l}$. Since $\p^{\out}(\mathcal{B}_N^{l}\cap \p^{\out}\mathcal{B}_{N})=(\mathcal{C}_{D_1})^c\cap \p^{\out}\mathcal{B}_{N}$ and because $B_r^\varepsilon(\xi_0)\subset \mathring D_1$, the path $p_{g,\tilde g}$ intersects $\p^{\out} \mathcal{C}_{B_{r_1}^\varepsilon(\xi_0)}$ and it follows that $$\# (\mathcal{B}_N^{l}\cap \mathcal{C}_{B^\varepsilon_{r_1}(\xi_0)})\geq \# (p_{g,\tilde g}\cap \mathcal{C}_{B^\varepsilon_{r_1}(\xi_0)})\geq  \bar n.$$ This contradicts the statement of corollary \ref{cor}.
\end{proof}

The following lemma states that for every $r>0$ there exists an $\tilde n\in \N$, such that the set $\mathcal{C}_{B^\varepsilon_{r}(\xi_0)}^{(\bar n \inn)}$ contains the truncated cone $\mathcal{C}_{B^\varepsilon_{r/2}(\xi_0)}\backslash \mathcal{B}_{\tilde n}$.

\begin{lemma}\label{neighborhood}
Let $B_r^{\varepsilon}(\xi_0)$ be a ball at infinity, $r_1$ as in (\ref{ri}) and define $r_0:=r_1/2$.
Then it holds for $\bar n$ defined in corollary \ref{cor} that
$$\mathcal{C}_{B^\varepsilon_{r_0}(\xi_0)}\backslash \mathcal{B}_{2\bar n}\subset \mathcal{C}_{B^\varepsilon_{r_1}(\xi_0)}^{(\bar n \inn)}.$$ 
\end{lemma}

\begin{proof}
Let $\bar n$ be as above and $m$ be an integer such that $m\geq \bar n.$ It follows by the definition of cones that if $g\in \mathcal{C}_{B^\varepsilon_{r_0}(\xi_0)}\backslash \mathcal{B}_{m}$, then $S(g)\subset B^\varepsilon_{r_0}(\xi_0)$. Let $\tilde \xi \in S(g)$ and let $\tilde g\in \tilde{\xi}_{v}$, such that $|\tilde g|=\bar n$. Then, by proposition \ref{shadow}, $S(\tilde g)\subset B^\varepsilon_{r_0+C_2e^{-\varepsilon \bar n}}(\xi_0)$. Let now $h\in (\mathcal{C}_{B_{r_1}^\varepsilon(\xi_0)})^c\cup \mathcal{B}_{\bar n}$. We will show that $d(g,h)\geq m-\bar n$. In case that $h\in \mathcal{B}_{\bar n}$, by the triangle inequality $d(g,h)\geq m- \bar n$. Assume now that $h\in (\mathcal{C}_{B_{r_1}^\varepsilon(\xi_0)}\cup \mathcal{B}_{\bar n})^c$. Then $|h|\geq \bar n$ and there exists a $\xi \in S(h)$ such that $\xi\notin B_{r_1}^\varepsilon(\xi_0)$. Let $\tilde h\in \xi_{v}$ be such that $|\tilde h|=\bar n$. It holds that $d^\varepsilon(\xi,\tilde \xi)\geq r_0-C_2e^{-\varepsilon \bar n}$ and it follows since $\bar n\geq k_3$ and by (\ref{n1}) that $\bar n\geq \varepsilon^{-1}\log(\pi^2r^{-1})$, so $C_2e^{-\varepsilon \bar n}\leq r_1/6$ and $$ \frac{1}{2}(r_0-C_2e^{-\varepsilon \bar n})   \geq C_2e^{-\varepsilon \bar n}.$$ By proposition \ref{shadow} it then follows that $S(\tilde g)\cap S(\tilde h)=\varnothing$ and in particular that $(\mathcal{U}_\xi \cap \mathcal{U}_{\tilde \xi})\backslash \mathcal{B}_{\bar n}=\varnothing$, so $d(\tilde g,\tilde h)\geq 2\delta$. It then easily follows by $\delta$-hyperbolicity that $$d(g,h)\geq m-\bar n.$$
Setting $m=2\bar n$, it holds that $d(g,h)\geq \bar n$, which implies that $g\in \mathcal{C}_{B^\varepsilon_{r_1}(\xi_0)}^{(\bar n \inn)}$ and finishes the proof.

\end{proof}

Now we are ready to prove the existence theorem stated in the introduction. Let us rephrase it slightly, before proving it:

\begin{theorema}
Let $j\in \{0,1\}$ and $\xi_j \in \mathring{D}_j$ be a point at infinity, let $r>0$ be such that $B^\varepsilon_r(\xi_j)\subset D_j$, $r_0=\frac{3r}{\pi^2}$, and let $\varepsilon>0$. Then there exists a constant $n_0\in \N$, such that %be as in corollary \ref{cor} and lemma \ref{neighborhood}, respectively. Then 
for all $N\geq n_0$ and 
for all $g\in \mathcal{C}_{B^\varepsilon_{r_0}(\xi_j)}\backslash \mathcal{B}_{n_0}$, $$|x_g^N-c_j|\leq \varepsilon.$$
%
%Let $\xi_0 \in \mathring{D}_0$ be a point in infinity and define the sequence of numbers $f(n):=\varepsilon^{-1}\cdot e^{-\varepsilon n}$. Let furthermore $\mathcal{B}^\varepsilon_{f(n)}(\xi_0)\subset \overline{\Gamma}$ be the sequence of balls in the visual metric converging to $\xi_0$. Then there exists a $\tilde n>0$, such that it holds for all $n\geq \tilde n$ and all $g\in \mathcal{B}^\varepsilon_{f(n)}(\xi_0)\backslash \p \Gamma$ that uniformly in $N$,
%$$|x_g^N-c_0|\leq \sigma_0\cdot k^{n-\tilde n}.$$
In particular, the global minimiser $\bar x=\lim_{k\to \infty}x^{N_k}$ constructed in proposition \ref{local_minimisers} solves the minimal Cauchy problem at infinity given by $D_0$ and $D_1$ (see definition \ref{dp_def}). 
\end{theorema}

\begin{proof}
As in the discussion above, let us first consider the case where $j=1$. Given an $\varepsilon>0$, let $\rho:=\min\{\rho_0,\varepsilon/2b\}$, $r_1=2r_0$ as in (\ref{ri}), and $n_0=2 \bar n$ as in corollary \ref{cor}. Then it holds by lemma \ref{neighborhood} that $\mathcal{C}_{B^\varepsilon_{r_0}(\xi_j)}\backslash \mathcal{B}_{n_0}\subset \mathcal{C}_{B^\varepsilon_{r_1}(\xi_0)}^{(\bar n \inn)}$ and by lemma \ref{values_at_infinity} that for all $N\geq n_0$, $\mathcal{C}_{B^\varepsilon_{r_1}(\xi_0)}^{(\bar n \inn)}\cap \mathcal{B}_N^l=\varnothing$, so $$(\mathcal{C}_{B^\varepsilon_{r_0}(\xi_j)}\backslash \mathcal{B}_{n_0})\cap \mathcal{B}_N^l=\varnothing.$$ This implies the statement of the theorem for $j=1$ by the definition of $ \mathcal{B}_N^l$ from lemma \ref{ts_estimate}. For the case $j=0$, an analogous proof can be followed, starting from proposition \ref{V} with $c_0$ and $c_1$ in the definitions appropriately exchanged.
\end{proof}

\appendix

\section{The Assouad dimension and growth of the graph}
In this section, we prove proposition \ref{goc}, i.e. we show that the growth of balls is bounded by an exponential related to the Assouad dimension of $\p \Gamma$. We assume throughout this appendix that $\varepsilon$ corresponding to a visual metric on $\Gamma$ is fixed, which makes $\p\Gamma$ into a metric space. 

Denote for any $B\subset \p \Gamma$ by $N_r(B)$ the smallest number of open sets of diameter $r$ required to cover $B$. The Assouad dimension of $\p \Gamma$ is then defined by 
\begin{equation}\label{AsD}d_A(\p\Gamma):=\inf\Big\{D \geq 0 \ | \ \exists C,\rho>0 : \forall 0<r<R\leq \rho, \sup_{\xi \in \p \Gamma} N_r(B_R^\varepsilon(\xi))\leq C\Big(\frac{R}{r}\Big)^D\Big\}
\end{equation} We refer the reader to \cite[Chapter 10]{heinonen} for an introduction to the Assouad dimension and its properties. Since $\Gamma$ is locally uniformly bounded it follows from \cite[Theorem 9.2]{BonkSchramm} that the Assouad dimension of $\p \Gamma$ is finite. As explained in \cite[Chapters 10]{heinonen}, this is equivalent to the fact that $(\Gamma,d_\varepsilon)$ is doubling as a metric space, i.e. there exists a number $M_d$, such that for every $r>0$ and $\xi_0\in \p\Gamma$, there exist $\{\xi_1,...,\xi_k\}$ for $k\leq M_d$, such that 
\begin{equation}\label{doubling_met}B_{2r}^\varepsilon(\xi_0)\subset \bigcup_{i=1}^k B_{r}^\varepsilon(\xi_i).\end{equation} 

Now we are ready to prove proposition \ref{goc}:

\begin{propositiona}[Growth of balls]
For every $D>d_A(\p\Gamma)$ there exists a constant $C_D$, such that for every $n\in \N$, $$\# \mathcal{B}_n\leq C_De^{\varepsilon D n} .$$
\end{propositiona}

\begin{proof}
It follows directly from the definition of the Assouad dimension that for every $D>d_A(\p\Gamma)$ and $r<\rho$, $N_r(B_\rho^\varepsilon(\xi))\leq C\rho^Dr^{-D}$. Since $\p \Gamma$ is bounded, by say $R_0$, we may replace $\rho$ in this inequality by $R_0$. Indeed, it holds for some $k\in \N$ that $2^k\rho\geq R_0$ and by the doubling condition $$N_r(B_{R_0}^\varepsilon(\xi))\leq M_d^k N_r(B_{\rho}^\varepsilon(\xi))\leq M_d^kC\rho^Dr^{-D}.$$

Let $s_n:=C_1 e^{-\varepsilon n}$, so that $N_{s_n}(B_{R_0}^\varepsilon(\xi))\leq M_d^kC2^D\rho^D e^{\varepsilon D n}$. Since there exists a covering of $\p \Gamma$ by $M_d^kC2^D\rho^D e^{\varepsilon D n}$ sets of diameter $s_n$, there also exists a covering of $\p \Gamma$ by the same number of balls $B^{\varepsilon}_{s_n}(\xi_i)$. By proposition \ref{shadow} it holds that for every $g\in \mathcal{S}_n\cap \xi_v$, where $\mathcal{S}_n$ is the sphere of radius $n$ in $G$ and $\xi\in \p \Gamma$, the shadow $S(g)$ has diameter bounded from below by $2C_1e^{\varepsilon n}$. Hence, there exists also a covering by shadows $S(g_i)$ with $g_i\in \mathcal{S}_n\cap (\xi_i)_v$ of size at most $M_d^kC\rho^D e^{\varepsilon D n}$. This means that every ray class $\xi_v$ passing through $\mathcal{S}_n$ is within a $\delta$ distance to one of $g_i$'s and furthermore, since $\Gamma$ is visual, that every $g$ is $\delta$-close to a ray. It follows by local uniform boundedness condition on $\Gamma$, that there exists a uniform constant $\tilde k$, such that for every $g\in G$, $\#\mathcal{B}_{2\delta}(g_i)\leq \tilde k$ holds and it follows that $\#\mathcal{S}_n\leq \tilde k M^kC\rho^D2^D e^{\varepsilon D n}$. The statement of the proposition now easily follows for an appropriately defined $C_D$ by $$ \#\mathcal{B}_n\leq \sum_{l=1}^n\# \mathcal{S}_l\leq  \tilde k M^kC\rho^D e^{\varepsilon D n}\sum_{k=1}^n e^{\varepsilon D k}\leq C_De^{\varepsilon D n}.$$
\end{proof}

\begin{small}
\bibliographystyle{amsplain}
\bibliography{AC-graphs}
\end{small}

\end{document}